\documentclass{amsart}

\usepackage[T1]{fontenc}
\usepackage[utf8]{inputenc}
\usepackage[USenglish]{babel}
\usepackage{textcase}
\usepackage{verbatim}
\usepackage{xcolor}

\usepackage[
	bookmarks=true,
	plainpages=false,
	linktocpage,
	colorlinks=true,
	citecolor=green!80!black,
	linkcolor=red!70!black,
	filecolor=magenta,
	urlcolor=magenta,
	breaklinks,
	pdfauthor={Martin Winter},
]{hyperref}

\usepackage{amsmath,amsthm}
\usepackage{calc, mathtools} %calc for \widthof, mathtools for \mathclap

\usepackage[
    colorinlistoftodos,
    backgroundcolor=yellow,
    textsize=footnotesize
]{todonotes}

\iftrue
\usepackage{amssymb} % incompatible with mathdesign
\else
\usepackage[charter,cal=cmcal]{mathdesign}
\fi

\usepackage{colortbl,color} % colors
\usepackage{array} % arrays
\usepackage{enumitem,moreenum} % numbered lists where each item can be labeled
\usepackage{cite} % for grouping references in the text which are adjacent in the bibliography
\usepackage[nameinlink,capitalize,noabbrev]{cleveref}
\usepackage{nicefrac}

\usepackage{tikz-cd} 

\usepackage[font=small,labelfont=bf]{caption}

\usepackage{blkarray}

\usepackage{inconsolata}

\newcommand{\NN}{\mathbb{N}}    % natural numbers            
\newcommand{\RR}{\mathbb{R}}    % real numbers                      
\def\^#1{^{(#1)}}
\def\s^#1{^{\smash{(#1)}}}

\def\:{\colon}

\def\nlspace{\nolinebreak\space}
\def\nls{\nlspace}

\newcommand{\precdot}{\mathrel{\prec\kern-0.9ex\cdot}}

\newcommand{\labelstyle}[1]{\upshape(\textit{#1})}
\newcommand{\mylabel}{\labelstyle{\roman*}}
\newenvironment{myenumerate}{\begin{enumerate}[label=\mylabel]}{\end{enumerate}}

\def\itm#1{{\labelstyle{\romannumeral#1\relax}}}

\newcommand{\freespace}{\kern.07em} 
\newcommand{\free}{\freespace\cdot\freespace} 

\newcommand{\enquote}[1]{``#1''}                                 % quotation marks

\newcommand{\sepline}{\par\smallskip\noindent\hrulefill\par\smallskip}

\newcommand{\TODO}{{\footnotesize\textcolor{red}{TODO}}}

\colorlet{mcolor}{red}
\colorlet{gcolor}{blue}

\newcommand{\msays}[1]{{\footnotesize\textcolor{mcolor}{\textbf M: #1}}}

\newtheoremstyle{mythmstyle} % name
    {\parsep}                    % Space above
    {\parsep}                    % Space below
    {\itshape}                   % Body font
    {}                           % Indent amount
    {\bfseries\scshape}          % Theorem head font
    {.}                          % Punctuation after theorem head
    {.5em}                       % Space after theorem head
    {}  % Theorem head spec (can be left empty, meaning ‘normal’)
\newtheoremstyle{mydefstyle} % name
    {\parsep}                    % Space above
    {\parsep}                    % Space below
    {}                   % Body font
    {}                           % Indent amount
    {\mdseries\scshape}          % Theorem head font
    {.}                          % Punctuation after theorem head
    {.5em}                       % Space after theorem head
    {}  % Theorem head spec (can be left empty, meaning ‘normal’)

\numberwithin{equation}{section}

\theoremstyle{theorem}
\newtheorem{theorem}{Theorem}[section]
\newtheorem{corollary}[theorem]{Corollary}
\newtheorem{lemma}[theorem]{Lemma}
\newtheorem{proposition}[theorem]{Proposition}

\newtheorem{main}{Theorem}

\theoremstyle{definition}

\newtheorem{example}[theorem]{Example}
\newtheorem{remark}[theorem]{Remark}
\newtheorem{notation}[theorem]{Notation}
\newtheorem{question}[theorem]{Question}
\newtheorem{observation}[theorem]{Observation}
\newtheorem{construction}[theorem]{Construction}

\crefname{theorem}{Theorem}{Theorems}
\crefname{proposition}{Proposition}{Propositions}
\crefname{lemma}{Lemma}{Lemmas}
\crefname{corollary}{Corollary}{Corollaries}
\crefname{remark}{Remark}{Remarks}
\crefname{example}{Example}{Examples}
\crefname{definition}{Definition}{Definitions}
\crefname{problem}{Problem}{Problems}
\crefname{observation}{Observation}{Observation}
\crefname{construction}{Construction}{Construction}

\theoremstyle{theorem}
\newtheorem{innercustomgeneric}{\customgenericname}
\providecommand{\customgenericname}{}
\newcommand{\newcustomtheorem}[2]{%
  \newenvironment{#1}[1]
  {%
   \renewcommand\customgenericname{#2}%
   \renewcommand\theinnercustomgeneric{##1}%
   \innercustomgeneric
  }
  {\endinnercustomgeneric}
}

\newcustomtheorem{theoremX}{Theorem}
\newcustomtheorem{lemmaX}{Lemma}
\newcustomtheorem{conjectureX}{Conjecture}

\DeclareMathOperator{\Aut}{Aut}

\DeclareMathOperator{\dist}{dist}

\DeclareMathOperator{\Sym}{Sym}
\DeclareMathOperator{\pot}{pot}  

\let\<=\langle
\let\>=\rangle

\def\...{...}
\newcommand{\shortStyle}{\textit}
\newcommand{\ie}{\shortStyle{i.e.,}}

\newcommand{\eg}{\shortStyle{e.g.}}

\newcommand{\wrt}{\shortStyle{w.r.t.}}

\newcommand{\cf}{\shortStyle{cf.}}

\newcommand{\resp}{resp.}

\newcommand{\as}{{a.s.}}

\makeatletter
\renewcommand*{\eqref}[1]{%
  \hyperref[{#1}]{\textup{\tagform@{\ref*{#1}}}}%
}
\makeatother

\newcommand{\ball}[3]{B_{#1}({#2},#3)}

\newcommand{\gDomainRange}{\RR_{\ge 0}}
\newcommand{\gFun}{$g\:\gDomainRange\to\gDomainRange$}

\renewcommand{\P}{\mathbb P}

\begin{document}

%%%%%%%%%%%%%%%%%%%%%%%%%%%%%%%%%%%%%%%%%%%%%%%%%%%%%%%%%%%%%%%%%%%%%%%%%%%%%%%%%%%%%%

\expandafter\title%[...]
{
%Unimodular Random 
(Random) Trees of Intermediate Volume Growth
}
		
\author[G. Kontogeorgiou]{George Kontogeorgiou}
\author[M. Winter]{Martin Winter}
\address{Mathematics Institute, University of Warwick, Coventry CV4 7AL, United Kingdom}
\email{martin.h.winter@warwick.ac.uk}
\email{george.kontogeorgiou@warwick.ac.uk}
	
% 05C05 - Trees
% 05C63 - Infinite graphs
% 05C80 - Random graphs (graph-theoretic aspects)
% 26A12 - Rate of growth of functions, orders of infinity, slowly varying functions
% 26A48 - Monotonic functions, generalizations
		
\date{\today}
\begin{abstract}
For every sufficiently well-behaved function $g\:\RR_{\ge 0}\to\RR_{\ge 0}$ that grows at least linearly and at most exponentially we construct a tree $T$ of~uniform volume growth $g$, that is,
$$C_1\cdot g(r/4)\le |\ball Gvr| \le C_2\cdot g(4r),\quad\text{for all $r\ge 0$ and $v\in V(T)$},$$ 
with $C_1,C_2>0$ and where $\ball Gvr$ denotes the ball of radius $r$ centered at a vertex $v$.
In particular, this yields examples of trees of uniform intermediate (\ie\ super-polynomial and sub-exponential) volume growth.

We use this construction to provide first examples of unimodular random rooted trees of uniform inter\-mediate growth, answering a question by Itai Benjamini. 
We find a peculiar change in structural properties for these trees at  growth $r^{\log\log r}$. 

Our results can be applied to obtain triangulations of $\RR^d$ for $d\ge 2$ with varied uniform growth behaviours, as well as Riemannian metric on $\RR^d$ for the same wide range of growth behaviors.
\end{abstract}

\maketitle

%\msays{Is it a good idea to say that we answer a question by Benjamini without ever having talked to him personally? Imagine you  asked a \enquote{stupid} question and someone attributes it to you (not that the question was stupid; just an extreme example).}

% \begin{mnote}
% For whom do we write this paper (or papers in general)? I say, we do \ul{not} write papers for the experts, and I have two reasons for this:
% %
% \begin{itemize}
%     \item many people learn at least some of their math from papers, especially if this stuff has not been put in a nice book form yet (\eg\ I haven't found a nice explanation of BS limits in a book, but actually, nor in a paper). If you are just learning the math, you are not an expert.
%     \item I hear papers being praised by saying \enquote{you should go read the original source, it explains it best}. This is clearly meant as a compliment, and it is said to non-experts.
% \end{itemize}
% \end{mnote}

%\msays{At the moment I have a strong feeling about the explanation of \enquote{uniform} in the abstract. Time is precious, attention spans are short, and often we (even more true for experts) have only time to scan the abstract (and not the definition section) to know whether the paper is relevant to us. The term \enquote{uniform} seems heavily overloaded, so I should see in a glimpse from the abstract whether the paper uses the notion of uniformity that is relevant to me.}

%\addtocounter{section}{-1}
\section{Introduction}

Given a simple graph $G$, a vertex $v\in V(G)$ and $r\ge 0$, the set 
\[\ball Gvr:=\{w\in V(G) \mid d_G(v,w)\leq r\}\]
is called the \emph{ball} of radius $r$ around $v$. 
The growth of the cardinality of these~balls~as $r$ increases is known as the \emph{growth behavior} or \emph{volume growth} of $G$ at the \mbox{vertex $v$}.
The two extreme cases of such growth are exhibited by two \mbox{instructive~examples}, the regular trees (of exponential growth) and the lattice graphs (of polynomial growth).
It is an ongoing endeavor to map the possible growth behaviors in various graph classes, the most famous example potentially being Cayley graphs of finitely generated groups, the central object of study in geometric group \mbox{theory (see \eg \cite{loh2017geometric}).
Ex}\-amples of major results in this regard are the existence of Cayley graphs of \emph{intermediate} growth (that is, super-polynomial but sub-exponential) \cite{grigorchuck}, and the proof that vertex-tran\-sitive graphs can have polynomial~growth for integer exponents only \cite[Theorem 2]{trofimov1985graphs}.

Cayley graphs (and more generally vertex-transitive graphs) automatically have the same growth at every vertex. 
In other graph classes this must be imposed~manually: we say that a graph $G$ is of \emph{uniform} growth if its growth does not vary too much~between vertices.
Following \cite{babaigrowth}, the precise formulation is as follows: there is a function \gFun\ and constants $c_1,C_1,c_2,C_2\in\RR_{>0}$ so that
\begin{equation}\label{eq:uniform_growth}
C_1\cdot g( c_1 r) \;\le\; |\ball Gvr| \;\le\;  C_2\cdot g( c_2 r)\quad\text{for all $r\ge 0$ and $v\in V(G)$}.
\end{equation}
The graph $G$ is then said to be of \emph{uniform growth $g$}.

In this article we construct infinite tree graphs of uniform volume growth for~a wide variety of growth behaviors, including intermediate growth and polynomial growth for non-integer exponents.
%Uniform growth hereby means that the growth of the balls varies not too much from vertex to vertex. A precise definition is given in the next subsection. As an example, vertex-transitive graphs, such as the regular trees, are always of uniform growth.

Subsequently we demonstrate how our construction gives rise to \emph{unimodular random rooted trees} for the same wide range of growth behaviour, answering~a~question by Itai Benjamini (private communication). 
We probe the structure~of these trees and find a threshold phenomenon happening roughly at the growth rate 
%\todo{\msays\ Mention relation to Cayley graph bounds}
$r^{\log\log r}$.
%qualitative differences between growth roughly below $r^{\log\log r}$ and above $r^{\alpha\log\log r}$ when $\alpha>1$. 
We identify unimodular trees of intermediate growth where only the root is random, \ie\ they are almost surely (\as) isomorphic to a particular deterministic tree. 

%{\color{red}

Historically, the interest in such trees has one of its origins in the curious observation, initially from physics, that planar triangulations can have non-quadratic uniform growth \cite{angel2003growth,ambj1997quantum}.
In their landmark paper \cite{benjamini2011recurrence} Benjamini and Schramm demonstrated how any tree of a particular growth can be turned into such a triangulation with a similar growth.
In the same paper they gave first examples for (unimodular) trees of uniform polynomial growth, also for non-integer exponents (it was later found that planar triangulation of sub-quadratic uniform growth are always \enquote{tree-like} \cite{benjamini2021triangulations}).
Following the construction in \cite{benjamini2011recurrence}, our trees yields planar triangulations for a wide range of growth behaviours.
Since our construction in particular yields one-ended trees, we actually obtain triangulations of the Euclidean plane, as well as of higher-dimensional Euclidean spaces. Those in turn correspond to Riemannian metrics on $\RR^d$ of the respective growth behavior.

Our work also follows a broader history of studies on the growth rate of trees. %, and particular\-ly of trees, on which we focus here.
%Additionally, there has also been a decent volume of work on trees of
%The first (unimodular) trees of uniform polynomial growth (with the goal of turning them into planar triangulations of arbitrary polynomial growth) were constructed by Benjamini and Schramm \cite{benjamini2011recurrence} (also for non-integer exponents).
Particular attention to exponential growth for trees was given by Timár \cite{timar2014stationary}, where the focus was on the existence of a well-defined basis for the exponential rate (which they call the \emph{exponential growth rate} of a the graph and which is distinct from our use of that term).
Quite recent advancements in this regard were made by Abert, Fraczyk and Hayes \cite{abert2022co}, who proved that this rate is well-defined in the case~of~uni\-modular trees.
Intermediate but not necessarily uniform growth in trees has been studied by Amir and Yang \cite{amir2022branching} as well as the references given therein.

\subsection{Main results}

%The main results in the present paper are the following:

We establish the existence of deterministic and unimodular~random rooted trees with volume growth $g$ for a wide variety of functions \gFun\ between (and including) polynomial and exponential growth.
The precise~statements are as follows: we first prove

\begin{main} \label{main}
    If \gFun\ is super-additive and (eventually) log-concave, then there exists a tree $T$ of uniform growth $g$. 
    %Let $g\:\RR_{\ge 0}\to\RR_{\ge 0}$ be a function that grows at least linearly and at most exponentially, and is in particular super-additive and log-concave. Then there exists a tree $T$ with uniform growth $g(r)$. 
\end{main}

This is subsequently generalized to

\begin{main} \label{unimain}
    If \gFun\ is super-additive and (eventually) log-concave, then there exists a unimodular random rooted tree $(\mathcal T,\omega)$ of uniform growth $g$. 
    %Let $g\:\RR_{\ge 0}\to\RR_{\ge 0}$ be a function that grows at least linearly and at most exponentially, and is in particular super-additive and log-concave. Then there is a unimodular random rooted tree $\mathcal{T}$ with uniform growth $g(r)$.
\end{main}

A sufficient introduction to random rooted trees and unimodular graphs is given in \cref{sec:unimodular} (or see \cite{aldouslyons}).

Super-additivity and log-concavity can be understood as formalizations~of~the~intuitive constraints that a locally finite graph grows \enquote{at least linearly} and \enquote{at~most exponentially}.
At the same time, these formulation preventing certain patho\-logies, such as too strong oscillations in the growth behavior.

We further prove a structure theorem (\cref{structure}) which provide more detail on the structure of the obtained unimodular trees $(\mathcal T,\omega)$ and how it depends on the prescribed growth rate. Its precise formulation requires some preparation, but the core message is as follows:
\begin{myenumerate}
    \item for uniform growth \emph{above} $r^{\alpha \log\log r}$ with $\alpha>1$ the constructed unimodular tree is (\as) reminiscent of the classic~canopy tree.
    \item for uniform growth \emph{below} $r^{\log\log r}$ the constructed unimodular tree has probability zero to coincide with any particular deterministic tree, \ie\ every countable set of rooted trees is attained with probability zero. %Also, it is \as\ 1-ended or 2-ended.
\end{myenumerate}
Except for almost linear growth, we find that $(\mathcal T,\omega)$ is \as\ 1-ended  (that is, any two~infinite rays are eventually identical).

\subsection{General notes on notation}
All graphs in this article are simple and poten\-tially infinite.
For a graph $G$ we write $V(G)$ for its vertex set and $E(G)$ for its edge set. For $v,w\in V(G)$ we write $vw\in E(G)$ for a connecting edge and $d_G(v,w)$ for their graph-theoretic distance in $G$.

For a graph of uniform growth as in \eqref{eq:uniform_growth} one generally distinguishes
\begin{itemize}
    \item \emph{uniform polynomial growth} if $g(r)=\exp(O(\log r))$,
    \item \emph{uniform exponential growth} if $g(r)=\exp(\Omega(r))$,
    \item \emph{uniform intermediate growth} if $g(r)=\exp(o(r))$ and $g(r)=\exp(\omega(\log r))$,
\end{itemize}
where we used the Landau symbols $O,\Omega, o,\omega$ as usual.

%We use the notation $\NN:=\{1,2,3,...\}$ and $\NN_0$ if $0$ is explicitly included.

%We should reasonably assume~$g(r)=\Omega(r)$ and $g=e^{\omega(r)}$, but~these and more general conditions will emerge naturally during the analysis of our construction. %As we see in a moment, we will also need to impose restrictions on the~os\-cil\-lation in the growth of $g$.

%\begin{itemize}
 %   \item $\log r=\log_2 r$ is logarithms to base 2. Otherwise write $\ln r = \log_e r$.
%\end{itemize}

\subsection{Overview}

%Starting from a suitable increasing function $g\:\RR_{\ge 0}\to\RR_{\ge 0}$, we achieve the following:
%\begin{enumerate}
    %\item we explicitly construct~a tree $T$ with uniform growth $g(r)$;%in which $|B_v^T(r)|$ grows approximately like $g(r)$, no matter~the~vertex \mbox{$v\in V(T)$}
    %\item we prove the existence of a unimodular tree $\mathcal{T}$ with uniform growth $g(r)$.
%\end{enumerate} 

\cref{sec:construction} provides the main construction: a recursively defined~sequence $T_n$ of finite trees as well as its limit tree $T$, which we later show to be~of~uniform growth. 
The growth of $T$ can be finely controlled using a sequence of parame\-ters $\delta_1,\delta_2,\delta_3,...\in\NN$.
We provide intuition for the connection between this~sequence and the growth of $T$, supplemented with several examples.
We also recall the construction from \cite{benjamini2011recurrence} for turning trees into triangulations of and Riemannian metric on $\RR^d,d\ge 2$ that inherit the growth behavior of the underlying tree.

In \cref{sec:main} we explain how to choose the $\delta_n$ to aim for a particular growth prescribed by some super-additive function \gFun.
We then~show~that,\nlspace subject to some technical conditions, $T$ is indeed~of~uniform growth $g$ (\cref{thm:main}):
$$C_1\cdot g(r/4)\le|\ball Gvr|\le C_2\cdot g(4r).$$
We conclude this section proving that the technical conditions are always satisfied if $g$ is (eventually) log-concave (\cref{res:conditions_in_g}).
This proves \cref{main}.

%We describe the general construction for the tree $T$ in \cref{sec:construction}. In \cref{sec:main}, we show that the volume growth of $T$ is uniform with bounds close to a specified target growth $g\:\RR_{\ge 0}\to\RR_{\ge 0}$. In particular, assuming mild and quite natural conditions on $g$, namely that it be super-additive and log-concave, we prove that $T$ has uniform growth $g(r)$.

In \cref{sec:unimodular} we recall the necessary terminology for unimodular random rooted graphs and Benjamini-Schramm limits. 
We investigate convergence (in the Benja\-mini-Schramm sense) of the sequence $T_n$ and find that its limit is indeed a unimo\-dular tree of uniform growth, proving \cref{unimain}.
We probe the structure of these limits in \cref{structure}.
We then construct a unimodular tree of uniform intermediate growth that is \as\ a unique deterministic tree (with randomly chosen root).

\section{The construction}
\label{sec:construction}

%\msays{Experimental alternative to the intro of \cref{sec:construction}. Work in progress.}

\par\medskip

% In this section we construct a tree $T$ and explain why we expect it to be 

% We construct a tree $T$ that we later show to be of uniform volume growth.
% More precisely, we construct a tree $T=T(\delta_1,\delta_2,...)$ for each integer sequence $\delta_1,\delta_2,\delta_3,...\in\NN$ with $\delta_n\ge 1$.

For each integer sequence $\delta_1,\delta_2,\delta_3,...\in\NN$ with $\delta_n\ge 1$ we construct a tree~$T=T(\delta_1,\delta_2,...)$.
The choice of sequence will determine the growth rate of $T$.
%
%We shall later see that the tree $T$ can manifest a wide range of growth behaviors depending on the input sequence, and we shall explain this connection in the next section.
%
%The growth behavior of $T$ expectedly depends on the sequence, and we later explain how to choose the $\delta_i$ to model a particular growth.
%
The tree $T$ is constructed as a limit object of the following sequence of trees $T_n$ where $n\ge 0$:

%We first introduce a recursively defined sequence of trees $T_n$ for $n\ge 0$ (see \cref{constr:T_n}) and subsequently define $T$ as a limit object of this sequence (see \cref{constr:T}).

%This tree will be obtained as the limit object of a recursively defined sequence of trees $T_n,n\ge 0$, which itself is built from the sequence $\delta_1,\delta_2,\delta_3,...$.

\begin{construction}\label{constr:T_n}
%We recursively define a sequence of trees $T_n,n\ge 0$ with two types of distinguished vertices, namely, a \emph{root} and a set of \emph{apocentric vertices}:
The trees $T_n$ are defined recursively.
%We recursively define a sequence of trees $T_n$ for $n\ge 0$. 
In each tree we distinguish two special types of vertices: a \emph{center}, and a set of so-called \emph{apocentric vertices} (or outer\-most or peripheral vertices), both will be defined alongside the trees:
%and in each tree two distinguished types of vertices, a \emph{root} and a set of \emph{apocentric vertices}:
%with two types of distinguished vertices: a \emph{root} and a set of so-called \emph{apocentric vertices}:
%Each $T_n$ has a \emph{root} and a set of so-called \emph{apocentric vertices} which are defined alongside the trees:
%
\begin{myenumerate}
    \item 
    $T_0$ is the tree consisting of a single vertex. This vertex is both the center of $T_0$ as well as its only apocentric vertex.
    
    \item 
    The tree $T_n$ is built from $\delta_n+1$ disjoint copies $\tau_0, \tau_1,...,\tau_{\delta_n}$ of $T_{n-1}$ that we join into a single tree by adding the following edges:
    %
    %where $\tau_0$ (which we call \emph{central copy}) will be connected to $\tau_1,...,\tau_{\delta_n}$ (which we call \emph{apocentric copies}) in the following way:
    %
    %for each $\tau_i,i\in\{1,...,\delta_n\}$ choose an apocentric vertex of $\tau_0$ and add an edge between this vertex and the root of $\tau_i$.
    for each $i\in\{1,...,\delta_n\}$ add an edge between the center of $\tau_i$ and some apocentric vertex of $\tau_0$.
    There is a choice in selecting these apocentric vertices of $\tau_0$ (and note that we can choose the same apocentric vertex more than once), but we shall require that these adjacencies are distributed in a maximally uniform way among the apocentric vertices of $\tau_0$ (we postpone a rigorous definition of this until we introduced suitable notation; see \cref{rem:max_uniform}).
    
    It remains to define the distinguished vertices of $T_n$: the center of~$T_n$~is the center of $\tau_0$; the apocentric vertices of $T_n$ are the apocentric vertices~of $\tau_1,...,$ $\tau_{\delta_n}$.

\end{myenumerate}
%
% We write $\tau\precdot T_n$ to denote that $\tau$ is such a copy. 
% This defines a relation on trees isomorphic to the $T_n$.
% Let $\prec$ be the transitive closure of $\precdot$, then we can write $\tau\prec T_n$ to denote the canonical copies of $T_m$ in $T_n$ for general $m< n$.
% We write $\tau\prec_m T_n$ to emphasize that $\tau$ is a copy of $T_m$.
%
See \cref{fig:construction} for an illustration of this recursive definition.
    
\begin{figure}[!h]
    \centering
    \includegraphics[width=0.85\textwidth]{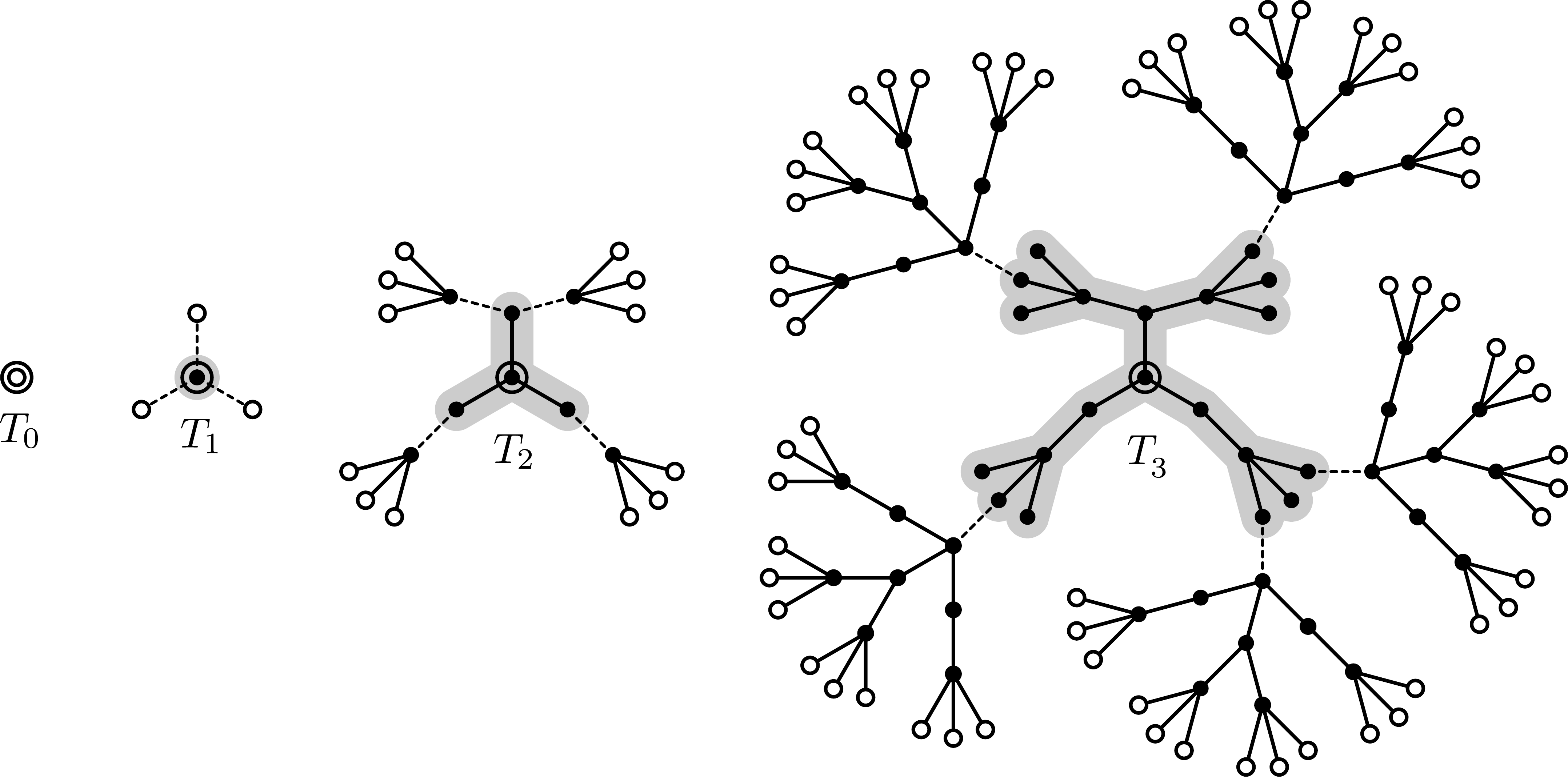}
    \caption{Illustration of the first four trees $T_0,...,T_3$ defined using the sequence $\delta_n:=n+2$. The ringed vertex is the center, and the white vertices are the apocentric vertices in the respective tree. The highlighted subgraph is the central copy $\tau_0$ in $T_n$. The dashed lines are the new edges added to connect the copies to form a single tree.}
    \label{fig:construction}
\end{figure}
\end{construction}

The following three properties follow immediately from the recursive definition:

\begin{observation}\label{res:T_n_properties} \quad
\begin{myenumerate}
    \item $T_n$ has exactly $(\delta_1+1)\cdots(\delta_n+1)$ vertices.
    \item $T_n$ has exactly $\delta_1\cdots\delta_n$ apocentric vertices, all of which are leaves of the tree (but not all leaves are necessarily apocentric).
    \item the distance from the center of $T_n$ to any of its apocentric vertices is $2^n-1$.
\end{myenumerate}
\end{observation}

One way to construct a \enquote{limit} of the $T_n$ is the following:

%The definition of $T:=T(\delta_1,\delta_2,...)$ as a limit object is straightforward.

\begin{construction}\label{constr:T}
For each $n\ge 1$ identify $T_n$ with one of its copies $\tau_0,\tau_1,...,\tau_{\delta_{n+1}}$ in $T_{n+1}$.
In this way we obtain an inclusion chain $T_0\subset T_1\subset T_2\subset\cdots$ and the~union $T=T(\delta_1,\delta_2,...):=\bigcup_{n\ge 0} T_n$ is an infinite tree.
\end{construction}

%\msays{either remove $T(d_1,d_2,...)$ everywhere, or define it.}

% In the following we shall write $T(\delta_1,\delta_2,...)$ to denote a limit tree obtained~from the sequence $\delta_1,\delta_2,\delta_3,...\in\NN$.
%
\iftrue % types of limit trees vs types of inclusion chains
For later use we distinguish three natural~types of limits:
%One natural way is to identify $T_n$ with its \enquote{central copy} $\tau_0\subset T_{n+1}$. The resulting inclusion chain will be called \emph{completely central}.
%Other conventions are conceivable, such as
%For later use we give names to the following other ways to define inclusion:
%
\begin{itemize}
    \item the \emph{centric limit} always identifies $T_n$ with the \enquote{central copy} $\tau_0$ in $T_{n+1}$.\nlspace
    This limit comes with a designated vertex $x^*\in V(T_0)\subset V(T)$, the \emph{global center}.
    \item \emph{apocentric limit}s always identify $T_n$ with an \enquote{apocentric copy} $\tau_i$ in $T_{n+1}$. We note that in this case the limit tree $T$ resembles the classicao canopy tree, in particular is \emph{one-ended}.
    \item \emph{mixed limits} make both central and apocentric identifications an infinite number of times each.
\end{itemize}
In \cref{sec:unimodular} we discuss a different and arguably more canonical way to take a limit (the Benjamini-Schramm limit) that avoids arbitrary identifications.

\else

For later use we distinguish three natural types of inclusion chains:
%One natural way is to identify $T_n$ with its \enquote{central copy} $\tau_0\subset T_{n+1}$. The resulting inclusion chain will be called \emph{completely central}.
%Other conventions are conceivable, such as
%For later use we give names to the following other ways to define inclusion:
%
\begin{itemize}
    \item \emph{centric} inclusion chains identify $T_n$ with the \enquote{central copy} $\tau_0$ in $T_{n+1}$.
    \item \emph{apocentric} inclusion chains identify $T_n$ with an \enquote{apocentric copy} $\tau_i$ in $T_{n+1}$.
    \item \emph{mixed} inclusion chains make both central and apocentric identifications infinitely often.
\end{itemize}
In \cref{sec:unimodular} we discuss a different and arguably more canonical way to make these identification via Benjamini-Schramm limits.

\fi

Our core claim is now that for \enquote{most} sequences $\delta_1,\delta_2,\delta_3,...\in\NN$ and independent of the type of the limit, the tree $T$ has a uniform volume growth of some sort, and that with a careful choice of the sequence we can model a wide range~of~growth~behaviors, including polynomial, intermediate and exponential.

%\msays{remove reference to global center.}

The following example computation gives us
%is meant to provide a first idea
a first idea of
%of 
the connection 
%that we should expect to find 
between the sequence $\delta_1,\delta_2,\delta_3,...\in\NN$ and the growth rate of $T$.
For this, let $T$ be the centric limit with global center $x^*\in V(T)$.
By \cref{res:T_n_properties}~\itm3 the ball of radius $r=2^n-1$ in $T$, centered at $x^*$, is exactly $T_n\subset T$.
By \cref{res:T_n_properties} \itm1 it follows
%
%This computation is based on three properties of the $T_n$ that follow immediately from the recursive definition:
%
%For later use we mention the following properties of $T_n$, all of which follow directly from the recursive definition:
%
%
% \begin{itemize}
%     \item the distance from the root of $T_n$ to any of its apocentric vertices is exactly $2^n-1$. In fact, $T_n\subset T$ is the ball of radius $2^n-1$ around the global root $x^*\in V(T)$.
%     \item $|T_n|=(\delta_1+1)\cdots(\delta_n+1)$ (where the empty product gives $|T_0|=1$).
% \end{itemize}
%
%In other words, for radii $r=2^n-1$ holds
%
\begin{equation}
\label{eq:ball_volume_in_T}
|\ball T{x^*\!}r| = |T_n| = (\delta_1+1)\cdots(\delta_n+1).    
\end{equation}
Thus, if we aim for, say, $\ball T{x^*\!}r\approx g(r)$ with a given growth function $g\:\NN_0\to\NN_0$, then \eqref{eq:ball_volume_in_T} suggests to use a sequence $\delta_1,\delta_2,\delta_3,...\in\NN$ for which 
%
% Thus, looking for 
% %Therefore, if we start from an increasing function $g\:\RR_{\ge 0}\to\RR_{\ge 0}$ and we want~to find 
% a sequence $\delta_1,\delta_2,\delta_3,...\in\NN$ for which $|\ball T{x^*\!}r| \approx g(r)$, where $g\:\NN\to\NN$ is some given growth function, \eqref{eq:ball_volume_in_T} suggests to use a sequence for which
$(\delta_1+1)\cdots(\delta_n+1)$ approximates $g(2^n-1)$.
In practice it turns out more convenient to approximate $g(2^n)$ (the computations are nicer and we still can prove uniform growth), and so we are lead to
%Thus, we are lead to
% Therefore, if we want $T$ to model a particular volume growth 
% %prescribed by a function $g\:\RR_{\ge 0}\to\RR_{\ge 0}$, that is, 
% $|\ball T{x^*}r| \approx g(r)$, then we should choose a sequence $\delta_1,\delta_2,\delta_3,...$ for which $(\delta_1+1)\cdots(\delta_n+1)$ approximates $f(n):=g(2^n-1)$.
%
\begin{equation} \label{eq:suggestions}
    \delta_n + 1\approx \frac{g(2^n)}{g(2^{n-1})},
\end{equation}
where we necessarily introduce an error when rounding the right side to an integer.
%
%In actuality we shall work with $g(2^n)$ instead of $g(2^n-1)$ because the~computations turn out nicer and we still can prove uniform growth.
%
%Our definition of \enquote{uniform growth} is flexible enough to allow for~this as any deviations can be absorbed into the constants.
%
%The We will still be able to prove uniform growth.
%
To establish uniform growth with a prescribed growth rate $g$ it remains to prove
\begin{itemize}
    %\item the error introduced by rounding $f(n)/f(n-1)$ is manageable.
    \item the error introduced by rounding the right side of \eqref{eq:suggestions} is manageable.
    \item an estimation close to \eqref{eq:ball_volume_in_T} holds for radii $r$ that are not of the form $2^n-1$.
    \item an estimation close to \eqref{eq:ball_volume_in_T} holds for general limit trees and around vertices other than a designated \enquote{global center}.
\end{itemize}
These points are addressed in the next section.

The remainder of this section is used to introduce helpful notation, to clarify the phrase \enquote{maximally uniform distribution of adjacencies} used in \cref{constr:T_n}, to provide examples, and to comment on the construction of triangulations of and Riemannian metric on $\RR^d,d\ge 2$ with prescribed growth.

%In the remainder of this section we address the \enquote{maximal uniform distribution of adjacencies} left undefined in \cref{constr:T_n} \itm2, and give examples that demonstrate the versatility of out construction.

\begin{notation}\label{not:copy}

By \cref{constr:T_n} \itm2 for every $n\in \mathbb{N}$, $T_n$ (and each tree isomorphic to $T_n$) comes with a canonical decomposition into copies of $T_{n-1}$.
Recursively we obtain a canonical decomposition of $T_n$ into copies of $T_m$ for each $m\le n$ (see \cref{fig:copies_of_T1}). We shall use the notation $\tau\prec_m T_n$ to indicate that $\tau$ is such a canonical copy of $T_m$, or $\tau \prec T_n$ if $m$ is not relevant.
%Lastly, by a slight abuse of notation, we may write \enquote{a copy $T_m\prec T_n$} to refer to an unnamed copy of $T_m$ inside of $T_n$.

A copy $\tau\prec T_n$ is called \emph{central} if it contains the center of $T_n$; it is called~\emph{apocentric} if it shares apocentric vertices with $T_n$ (and one can easily show that then \emph{all}~apocentric vertices of $\tau$ are apocentric in $T_n$).

Finally, note that by way of construction, a limit tree $T$ in the sense of \cref{constr:T} has such a canonical decomposition too, and the notion $\tau\prec T$ therefore makes sense as well.
\end{notation}

\begin{figure}[h!]
    \centering
    \includegraphics[width=0.47\textwidth]{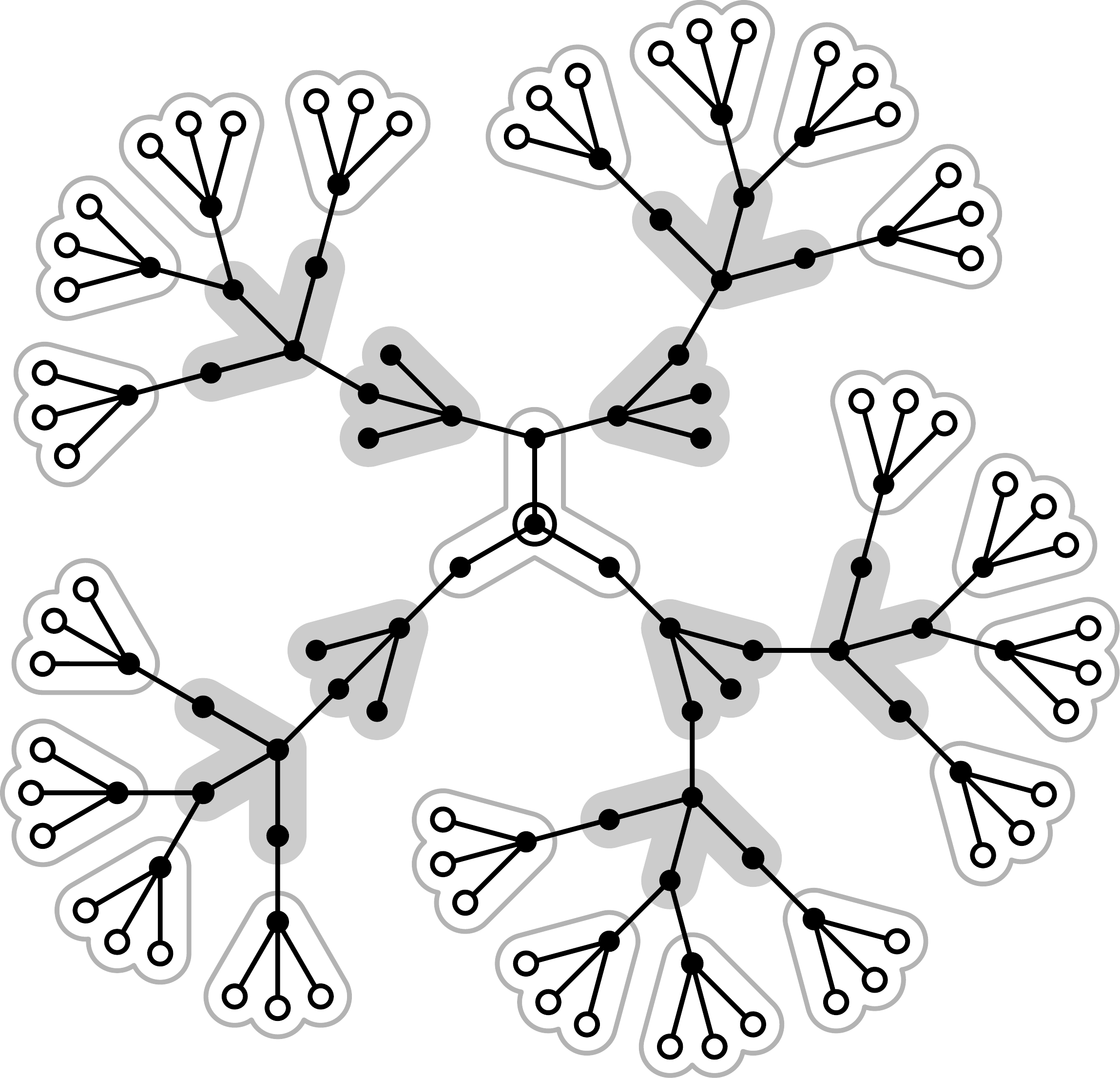}
    \caption{All canonical copies $\tau\prec_1 T_3$ are highlighted. The central~and apocentric copies are highlighted in white, the others in gray. %Observe that $T_3$ contains further subgraphs isomorphic to $T_1$ that are not copies.
    }
    \label{fig:copies_of_T1}
\end{figure}

% \begin{remark}\label{rem:copy}
% The symbol $\precdot$ defines a relation between (graphs isomorphic to) the $T_n,n\ge0$.
% Let $\prec$ be its transitive closure.
% If $\tau\prec T_n$ we say \,\enquote{$\tau$ is a copy in $T_n$}, and we write $\tau\prec_m T_n$ to emphasize that $\tau$ is a copy of $T_m$.
% We call $\tau\prec T_n$ a \emph{central copy} if it contains the root of $T_n$, we call it an \emph{apocentric copy} if it shares apocentric vertices with $T_n$ (and it is not hard to see that then all apocentric vertices of $\tau$ are also apocentric in $T_n$).
% \end{remark}

% The \enquote{maximally uniform distribution of adjacencies} between the apocentric vertices of $\tau_0$ and the roots of $\tau_1,...,\tau_{\delta_n}$ in \cref{constr:T_n} \itm2 will later be crucial for establishing uniform growth of $T$ (see \cref{thm:main}).
% A formal definition is given below.

With this notation in place we can clarify our use of \enquote{maximal uniform~distribution of adjacencies} in \cref{constr:T_n} \itm2.

\begin{remark}\label{rem:max_uniform}
%\enquote{Maximal uniform distribution of adjacencies} roughly means that certain \enquote{important} subsets of apocentric vertices of $\tau_0$ are involved in the expected number of \enquote{outwards adjacencies} to the apocentric copies $\tau_1,...,\tau_{\delta_n}$.

%To understand what is the \enquote{expected number}, note that $T_n$ has exactly $\delta_1\cdots\delta_n$ apocentric vertices.
%We make this precise.

Let $\tau_0\prec_{n-1} T_{n}$ be the central copy.
%An \emph{outwards edge} is an edge connecting $\tau$ with an apocentric copy $T_n\prec T_{n+1}$. Roughly, \enquote{maximal uniform distribution of adjacencies} means that the end vertices of these edges that lie in $\tau$ are distributed \enquote{as if they were chosen at random}.
Then there are exactly $\delta_n$ \enquote{outwards edges} connecting $\tau_0$ to the apocentric copies $\tau_1,...,\tau_{\delta_n}\prec_{n-1} T_n$.
%
%More formally, let $\mathcal E$ be the set of outwards edges as defined above.
%For each apocentric copy $\tau'\prec_m \tau$ with $m\le n$ let $\mathcal E'\subseteq\mathcal E$ be the subset of outwards edges that end in $\tau'$.
%Then maximally uniform means that $\mathcal E'$ contains very close to the number of edges that one would expect, namely
\enquote{Maximal uniform distribution} means that every apocentric copy $\tau\prec_m \tau_0$ (where $m\le n-1$) intersects the expected number of these edges (up to rounding). More precisely, if $E_\tau$ is the number of the \enquote{outwards edges} that have an end in $\tau$, then
\begin{equation}\label{eq:max_uniform}
\Big\lfloor \, \frac{\delta_1\cdots\delta_m}{\delta_1\cdots\delta_{n-1}} \cdot \delta_n \Big\rfloor
\le
E_\tau
\le 
\Big\lceil \, \frac{\delta_1\cdots\delta_m}{\delta_1\cdots\delta_{n-1}} \cdot \delta_n \Big\rceil,
\end{equation}
where $(\delta_1\cdots\delta_m)/(\delta\cdots\delta_{n-1})$ is exactly the fraction of apocentric vertices of $\tau_0$ that are also in $\tau$ (\cf\ \cref{res:T_n_properties} \itm2).

It is not hard to see that in each step of \cref{constr:T_n} this distribution can be achieved by adding the \enquote{outwards edges} one by one. % (\eg\ choose the next apocentric vertex of $\tau_0$ maximally distant from all previously chosen apocentric vertices).
The reader can verify that the steps shown in \cref{fig:construction} are in accordance with~a~max\-imally uniform distribution.

%Maximal uniform distribution of the connections from the apocentric vertices of $\tau$ to the apocentric copies of $T_{n}\prec T_{n+1}$ means, roughly, that the 

% We give a rigorous definition of what we mean by \enquote{maximally uniform distribution of adjacencies}.

% Let $\tau\prec_n T_{n+1}$ be the central copy of $T_n$.
% Let $a(\tau)\in\NN$ be its number of apocentric vertices and $a_+(\tau)\le a(\tau)$ the number of \emph{outwards vertices} among them, that is, apocentric vertices that are adjacent to apocentric copies $T_n\prec T_{n+1}$. 
% \enquote{Maximally uniform distribution of connections} means that for each apocentric copy $\tau' \prec_m \tau$, where $m\le n$, the number $a_+(\tau')$ of outwards vertices in $\tau'$ is maximally close to what we would expect if the adjacencies were distributed at random:
% %
% $$
% \Big\lfloor\, \frac{a_+(\tau)}{a(\tau)}\cdot a(\tau')\, \Big\rfloor
% \le a_+(\tau')
% \le \Big\lceil\, \frac{a_+(\tau)}{a(\tau)} \cdot a(\tau') \, \Big\rceil.
% $$
\end{remark}

We provide three examples demonstrating the versatility~of~\cref{constr:T}.
%For conciseness we use $\bar g(r):=g(r-1)$, then $f(n)=\bar g(2^n)$.

\begin{example}[Polynomial growth]\label{ex:polynomial}
If we aim for polynomial growth $g(r)=r^\alpha,\alpha\in\NN$ then the heuristics \eqref{eq:suggestions} suggests to use a constant sequence $\delta_n:=2^\alpha-1$.

In fact, the corresponding trees $T_n$ embed nicely into the $\alpha$-th power%
\footnote{Recall, the $\alpha$-th power of $G$ is a graph $G^\alpha$ with vertex set $V(G)$ and an edge between any to vertices whose distance in $G$ is at most $\alpha$.}
of~the~$\alpha$-dimensional lattice graph (shown in \cref{fig:2_lattice_construction} for $\alpha=2$). %, giving further evidence for polynomial volume growth.
\begin{figure}[h!]
    \centering
    \includegraphics[width=0.92\textwidth]{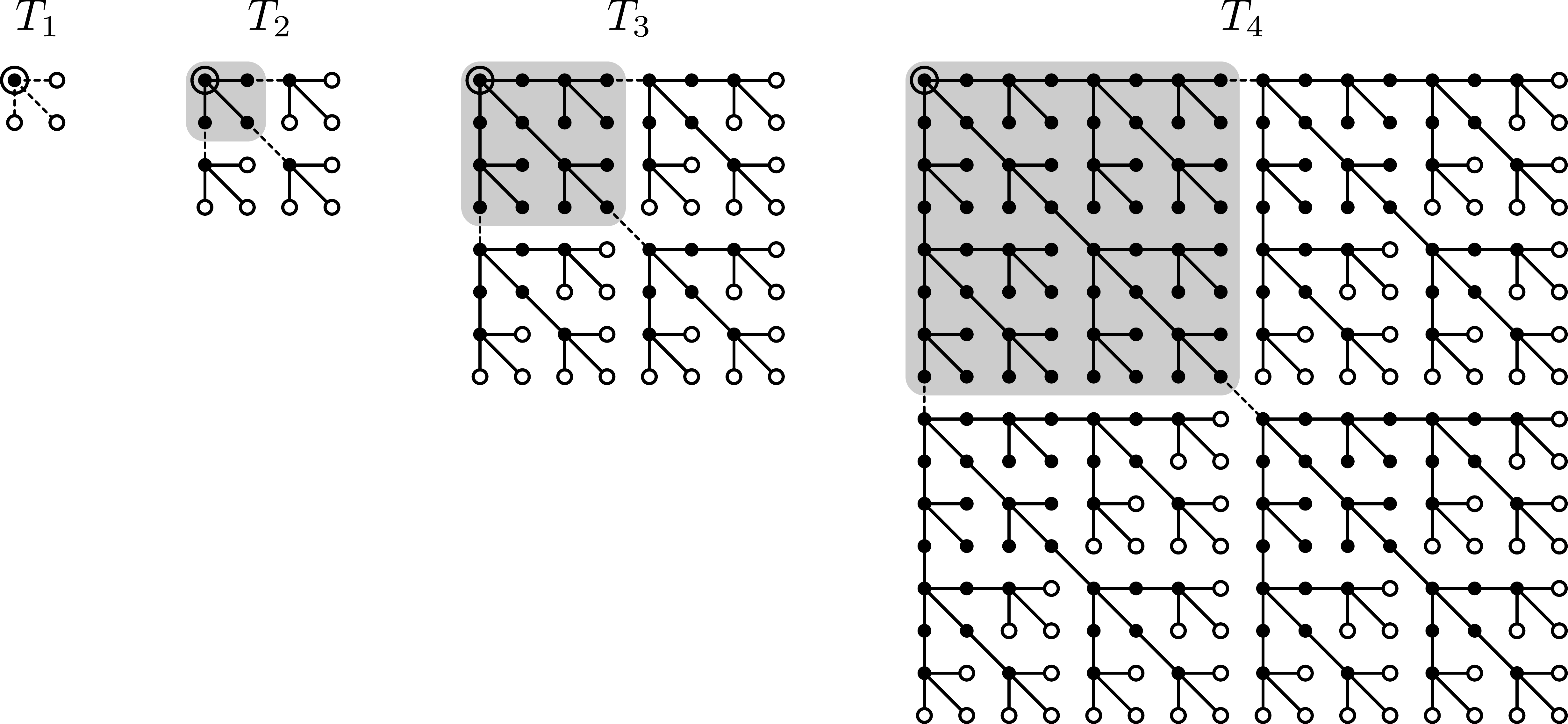}
    \caption{
        Embedding of the tree $T(3,3,...)$ into the square of the 2-dimensional lattice graph.
    }
    \label{fig:2_lattice_construction}
\end{figure}

More generally, for any constant sequence $\delta_n:=c$ we expect to find polynomial volume growth, potentially with a non-integer exponent $\log(c+1)$.
\end{example}

\begin{example}[Exponential growth]\label{ex:exponential}
%The $d$-regular tree has growth function $\bar g(r)=d(d-1)^r$. This yield $f(n)=d(d-1)^{2^n}$ and $\delta_n=(d-1)^{2^{n-1}}$.
For the sequence $\delta_n:=d^{\kern0.15ex 2^{n-1}}\!,d\in\NN$ the~centric limit $T$ is the $d$-ary tree,
%Replacing $d_1:=d+2$ will result in the $(d+1)$-regular tree.
in particular, of exponential volume growth (see~\cref{fig:binary_tree} for the case $d=2$, \ie\ the binary tree).
In fact, using \eqref{eq:ball_volume_in_T} for $r=2^n$ (where $x^*\in V(T)$ is the global center) we find
 $$|\ball T{x^*\!}{r-1}| = (\delta_1+1)\cdots(\delta_n+1)=\prod_{k=1}^n \Big(d^{\kern0.15ex 2^{k-1}}\!\!+1\Big) = \sum_{i=0}^{\mathclap{2^n-1}} d^{\kern0.15ex i} = \frac{d^{\kern0.15ex 2^n}-1}{d-1} = \frac{d^{\kern0.15ex r}-1}{d-1}.$$
\begin{figure}[h!]
    \centering
    \includegraphics[width=0.85\textwidth]{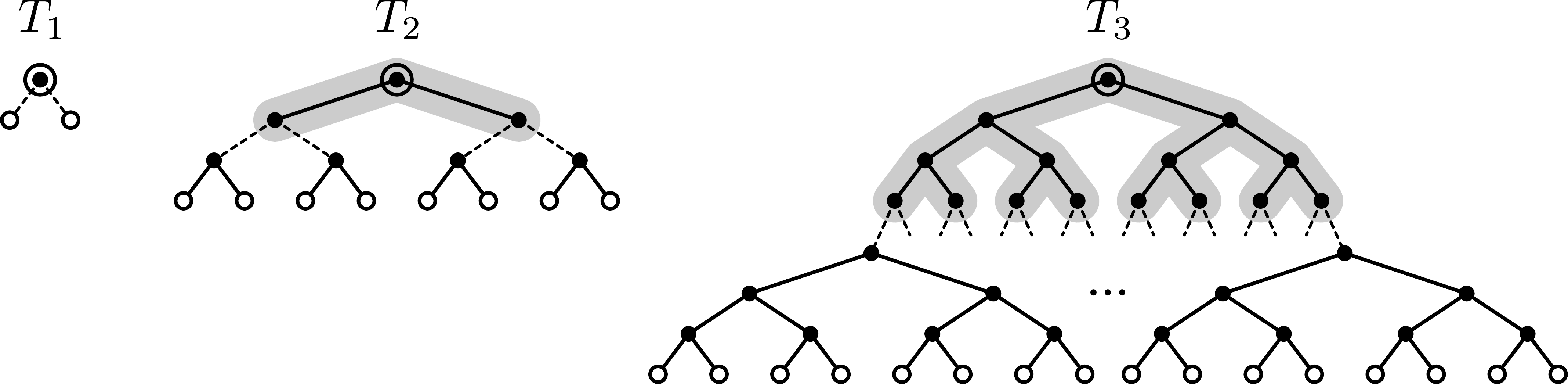}
    \caption{
        The binary tree as constructed from \cref{constr:T_n} using the doubly exponential sequence $\delta_n=2^{2^{n-1}}$.
    }
    \label{fig:binary_tree}
\end{figure}
\end{example}

Extrapolating from \cref{ex:polynomial} and \cref{ex:exponential}, it seems reasonable that unbounded sequences $\delta_1,\delta_2,\delta_3,...$ with a growth sufficiently below doubly exponential result in intermediate volume growth.

\begin{example}[Intermediate growth]\label{ex:loglogr}
For $\delta_n:=(n+3)^\alpha -1$, $\alpha\in\NN$ we can compute this explicitly (see \cref{fig:construction} for the case $\alpha=1$).
If $T$ is the centric limit with global center $x^*\in V(T)$ and $r=2^n$, then:
\begin{align*}
%\bar g(r) 
|\ball T{x^*\!}{r-1}|
= (\delta_1+1)\cdots(\delta_n+1) 
&= (\tfrac16 (n+3)!)^\alpha 
 \sim (n! n^3)^\alpha
 \sim (n^n e^{-n} n^{7/2})^\alpha 
\\ &= r^{\alpha \log\log r} r^{-\alpha/\ln 2}\, (\log r)^{7\alpha/2}
%  \sim \tfrac16(n! n^3)^\alpha
%  \sim \tfrac16(n^n e^{-n} n^{7/2})^\alpha 
% \\ &= \tfrac16r^{\alpha \log\log r} r^{-\alpha/\ln 2}\, (\log r)^{7\alpha/2}
\end{align*}
We therefore expect this choice of sequence to lead to a tree of uniform intermediate volume growth.
Trees constructed from $\delta_n\sim n^\alpha$ present an interes\-ting~boundary case in \cref{sec:unimodular} when we discuss unimodular random trees (see also \cref{structure}).
% \begin{align*}
% |\ball T{x^*}{r-1}| = |T_n|= \tfrac16 (n+3)! 
% \, &\sim\,
% n! \, n^3  
% \,\sim\,
% % n^n \cdot n^{9/2}e^{-n}
%  n^n e^{-n} \sqrt n \, n^3 
% \\ &= \,
% r^{\log\log r}\cdot
% r^{-1/\ln 2} \cdot (\log r)^{7/2}  .
% %\cdot q\big(\!\sqrt{\log r}\kern0.75pt\big)
% \end{align*}
\end{example}

%\hrulefill

We close this section with a brief discussion of how to turn these trees into planar triangulations and Riemannian manifolds (following the construction in \cite{benjamini2011recurrence}):

\begin{construction}\label{constr:triangulations}
Let $T$ be a tree of maximum degree $\Delta$.
We fix a triangulated sphere with at least $\Delta$ pairwise disjoint triangles. We take copies of this sphere, one for each vertex of $T$, and identify two spheres along a triangle when the associated vertices are adjacent in $T$.
This yields a planar triangulation.
If $T$ is of uniform growth $g$, so is this triangulation.
\end{construction}

Using trees with a variety of growth rates, we can build planar triangulations with the same wide range of growth behaviors.

It turns out, we can say more: previously known triangulations of polynomial growth (such as in \cite{benjamini2011recurrence}) are planar, but are not necessarily triangulations of the~plane, \ie\ they are not necessarily homeomorphic to $\RR^2$. 
For this to be the case, the tree $T$ needs to be \emph{one-ended}.
In fact, this is the case for the tree~$T$~obtained~as~the~apo\-centric limit via \cref{constr:T}.

It is straightforward to modify \cref{constr:triangulations} to yield triangulations of higher-dimensional Euclidean spaces or even Riemannian metrics on $\RR^d$ exhibiting a particular volume growth.
For the triangulation, start from a simplicial $d$-sphere with at least $\Delta$ pairwise disjoint simplices and proceed by gluing them along simplices as in \cref{constr:triangulations} to obtain a triangulation of $\RR^d$.
For the Riemannian metric, replace each vertex of $T$ by a unit sphere and connect the spheres smoothly along the edges of $T$ using sufficiently thin cylinders of bounded length. The result is diffeomorphic to $\RR^d$ but with a metric that inherits its growth from $T$.
%Likewise we can construct Riemannian metrics on $\RR^n$ turning it into a Riemannian manifolds of a desired growth: \TODO

%By choosing a suitable metric on each simplex (finite volume and diameter, and so to smooth out the transition at the boundaries between simplices), this turns $\RR^d$ into a Riemannian manifold of the respective volume growth.

%Moreover, since the barycentric subdivision of the cartesian product of a triangulation of $\RR^2$ with a Riemannian $d$-manifold is itself a triangulation of a Riemannian $(d+2)$-manifold, we can construct triangulations of any dimension exhibiting a wide range of growth.\

%From this we obtain triangulations of the plane for a wider range of growth behaviors.
%The same analysis \cref{thm:main} of the growth behavior applies to this tree and so we obtain triangulations of the plane with arbitrary uniform growth.

\section{Uniform volume growth}
\label{sec:main}

We fix an increasing function \gFun\ that shall serve as our target growth rate. %; we also set $\bar g(r)$ $:=g(r-1)$.
The goal of this section is to construct an appro\-priate sequence $\delta_1,\delta_2,\delta_3,...\in\NN$ for which we can show that $T:=T(\delta_1,\delta_2,...)$ is of uniform volume growth $g$.\nlspace  
%\textcolor{lightgray}{
We shall show in particular that this is feasible if $g$ is super-additive and log-concave, thus proving \cref{main}.
%}

Recall that the uniform growth rate of a graph must be at least linear (like in~a path graph) and at most exponential (like in a regular tree).
It would therefore~be reasonable to impose these constraints on the target growth rate $g$ right away.\nlspace
%As $g$ is supposed to be the uniform growth rate of a locally finite graph, intuitively we should require that $g$ grows at least linearly and at most exponentially.
However, we shall not do this: these constraints will also emerge naturally during our journey towards the main theorem of this section.
In fact, we shall uncover even more precise constraints expressed in terms of the sequence $\delta_1,\delta_2,\delta_3,...\in\NN$ (see \cref{rem:c,rem:dbar}) and we discuss how they relate to the natural constraints of $g$ growing at least linear and at most exponential.
% (see \cref{rem:c} and \cref{rem:dbar}, as well as \cref{thm:main})
%These can often be expressed more directly in terms of the sequence of $\delta_n$, but we make an effort to provide versions in terms of $g$ as well.

Following our motivation from the last section we firstly try to approximate $g(2^n)$ by $(\delta_1+1)\cdots(\delta_n+1)$.

\begin{lemma}\label{res:d_i}
There exists a sequence $\delta_1,\delta_2,\delta_3,...\in\NN$ so that
\begin{equation}\label{eq:d_bounds}
(1-c)\cdot g(2^n)  \;<\; (\delta_1+1)\cdots(\delta_n+1) \;<\; (1+c)\cdot g(2^n),
\end{equation}
for all $n\ge 0$, where $c:=\sup_n g(2^n)/g(2^{n+1}) \le 1$.
\begin{proof}
For this proof we abbreviate $\bar\delta_n:=\delta_n+1$. 
We set $\bar\delta_1 := g(2)$ and recursively %\msays{$f(1)$ might not be an integer}
$$\bar\delta_{n+1}:=\begin{cases}
\lceil g(2^{n+1})/g(2^n)\rceil & \text{if $\bar\delta_1\cdots\bar\delta_{n}\le g(2^n)$}\\
\lfloor g(2^{n+1})/g(2^n)\rfloor & \text{if $\bar\delta_1\cdots\bar\delta_{n}> g(2^n)$}
\end{cases}.
$$

This sequence satisfies \eqref{eq:d_bounds} as we show by induction on $n$: the induction base is clear from the definition of $\bar\delta_1$. Next, if \eqref{eq:d_bounds} holds for $n$, and $\bar\delta_1\cdots\bar\delta_{n}\le g(2^n)$, then
\begin{align*}
\bar\delta_1\cdots \bar\delta_n \cdot\bar\delta_{n+1} &\le g(2^n) \cdot\Big\lceil \frac{g(2^{n+1})}{g(2^n)} \Big\rceil 
< g(2^n)\cdot \Big( \frac{g(2^{n+1})}{g(2^n)} + 1\Big)
\\&= \Big( 1+ \frac{g(2^n)}{g(2^{n+1})} \Big) \cdot g(2^{n+1}) \le ( 1+ c )\cdot  g(2^{n+1}), \\[1ex]
\bar\delta_1\cdots \bar\delta_n \cdot\bar\delta_{n+1} 
&> (1-c)\cdot g(2^n)\cdot \Big\lceil \frac{g(2^{n+1}}{g(2^n)} \Big\rceil 
\ge (1-c)\cdot g(2^n)\cdot \frac{g(2^{n+1})}{g(2^n)}
\\&=(1-c)\cdot g(2^{n+1}).
\end{align*}
An analogous argument applies in the case $\bar\delta_1\cdots\bar\delta_{n}> g(2^n)$.
\end{proof}
\end{lemma}

%In the following let $\delta_1,\delta_2,\delta_3,...\in\NN$ be the sequence constructed in the proof of \cref{res:d_i}.

\begin{remark}\label{rem:c}
Recall that \cref{constr:T_n} requires $\delta_n \ge1$ for all $n\ge 1$ in order for $T(\delta_1,\delta_2,...)$ to be well-defined.
The sequence provided by the proof of \cref{res:d_i} does not generally have this property unless further requirements on $g$ are met.
One possible condition is \emph{super-additivity}, which turns out to also formalize the idea of \enquote{$g$ growing at least linearly}.

Super-additivity means $g(r_1+r_2)\ge g(r_1)+ g(r_2)$ and indeed implies $g(2^{n+1}) = g(2^n+2^n)\ge 2g(2^n)$, and hence
%
%More precisely, the following is sufficient: $g(1)\ge 2$ and $\bar g$ is \emph{super-additive}, \ie\ $\bar g(r_1+r_2)\ge \bar g(r_1)+\bar g(r_2)$.
%Then $\bar g(2r)\ge 2\bar g(r)\implies f(n+1)\ge 2f(n)$, from which follows
%
$$\delta_{n+1}+1\ge \Big\lfloor\frac{g(2^{n+1})}{g(2^n)}\Big\rfloor \ge 2.$$
%
%In the following we shall therefore assume that $g$ is super-additive.
%This formalizes \enquote{$g$ should grow at least linearly}, and also yields
%This implies 
This moreover implies $c:=\sup_n g(2^n)/g(2^{n+1}) \le 1/2 <1$ and guarantees that the lower bound in \cref{res:d_i} is not vacuous.
%
%Together with $g(0)=1$, these are the only requirements on the growth of $g$ from below.
%This also captures the intuition that $g$ should grow \enquote{at least linearly}.
\end{remark}

%In the following we assume that $g$ is super-additive to ensure that $T$ exists.

\begin{remark}\label{rem:dbar}
%\Cref{constr:tree} provide a reason to put upper bounds on the growth of $g$:
%An upper bound on the growth of $g$ emerges as necessary~if~we~investigate the vertex degrees in $T$.
Recall that graphs of uniform growth are necessarily of bounded~degree.
But again, the sequence constructed in the proof of \cref{res:d_i} does not~necessarily produce such a tree unless further requirements on $g$ are met.
We elaborate how these requirements can be interpreted as \enquote{$g$ growing at most exponentially}.
%
%Making sure that \cref{constr:T} yields such a graph leads us to a more precise formulation of \enquote{$g$ should grow at most exponentially}.
% by the fact that graphs of uniform growth are necessarily of finite maximum degree.

The tree $T_{n-1}$ has exactly $\delta_1\cdots\delta_{n-1}$ apocentric vertices. That is, each apocentric vertex of the central copy $\tau_0\prec_{n-1}T_n$ is adjacent, on average,\nlspace to
\begin{equation}\label{eq:Delta}
\Delta(n):= \frac{\delta_n}{\delta_1\cdots\delta_{n-1}}
\end{equation}
apocentric copies $\prec_{n-1} T_n$.
By the maximal uniform distribution of adjacencies (\cf\ \cref{rem:max_uniform}) the maximum degree among the apocentric vertices of $\tau_0$ is exactly $\lceil \Delta(n) \rceil + 1$.
Thus, since $T$ contains $T_n$ for each $n\ge 0$, $T$ is of bounded degree~only if $\bar \Delta:=\sup_n \lceil\Delta(n)\rceil + 1<\infty$.
We can use this to put an upper bound on the growth of the sequence $\delta_1,\delta_2,\delta_3,...$: for $n\ge 2$ and using \eqref{eq:Delta} it holds
%Expanding \eqref{eq:Delta} under the constraint $\Delta(n)\le\bar \Delta-1$ yields
%One can expand \eqref{eq:Delta} under this bound  to obtain % that this bound together with \eqref{eq:Delta} yields
%
$$
\frac{\delta_n}{\delta_1\cdots\delta_{n-1}}\le\bar\Delta - 1
\;\implies \;
\delta_n\le \big((\bar \Delta-1)\delta_1\big)^{2^{n-2}}.
$$
Comparing with \cref{ex:exponential} suggests that $g$ grows at most exponentially. 

%Besides bounding the growth of $g$, putting an upper bound on $\Delta(n)$ has further more subtle effects, such as preventing strong oscillations in the growth of $g$ (see~\eg\ \cref{res:conditions_in_g}).

% More precisely, the maximum degree among apocentric vertices of $T_{n-1}$ (when viewed as vertices in $T$) is $\lceil\Delta(n)\rceil+1$.
% Since a graph of uniform volume growth must be locally finite, we found that $\Delta$ must be bounded.
% This imposes a first necessary upper bound on the growth of $g$: roughly, this corresponds to $g$ growing at most exponentially; but it is more subtle in that it also prevents strong oscillations in the growth of $g$.

% However, $\Delta$ being bounded is not quite sufficient for uniform volume growth.
% In the following main theorem we show that a comparatively simple sufficient criterion can be nicely stated in terms of $\Delta$ as well.
\end{remark}

It remains to prove the main result of this section: establishing the exact growth rate of $T=T(\delta_1,\delta_2,...)$ in terms of $g$.
As we shall see, bounded degree (as discussed in \cref{rem:dbar}) is not sufficient to prove uniform volume growth. Instead we require an additional criterion that can also be concisely expressed in terms of the function $\Delta$.
In terms of $g$ this can be interpreted as preventing too strong oscillations in the growth of $g$ (\cf\ \cref{res:conditions_in_g}).

In order to state the main result we recollect that \gFun\ is a strictly~in\-creasing super-additive function.
The sequence $\delta_1$,\nlspace $\delta_2$, $\delta_3,...$ $\in\NN$ is chosen to satisfy
$$
(1-c)\cdot g(2^n)  \;<\; (\delta_1+1)\cdots(\delta_n+1) \;<\; (1+c)\cdot g(2^n),
$$
for some $c\le 1/2$ (according to \cref{res:d_i}) and satisfies $\delta_n\ge 1$ by \cref{rem:c}. In particular, $T=T(\delta_1,\delta_2,...)$ exists.
Finally, $\Delta(n)$ is defined as in \eqref{eq:Delta}. The~main~result then reads as follows:

\begin{theorem}\label{thm:main}
Fix $v\in V(T)$ and $r\ge 0$. Then holds
\begin{myenumerate}
\setlength\itemsep{1ex}
\item \, \\ \,
\vspace{-2.325\baselineskip}
\begin{align*}
\mathrlap{
    |\ball Tvr| \,\ge\, (1-c) \cdot g(r/4) 
                \,\ge\, \tfrac12\, g(r/4).
} \hspace{16em}
\end{align*}
%
%where $r/4$ in the argument of $g$ is rounded down to the largest integer.
 
\item If $\bar\Delta:=\smash{\displaystyle\sup_n}\lceil\Delta(n)\rceil +1<\infty$, then $T$ is of maximum degree $\bar\Delta$ and
\begin{align*}
\mathrlap{
    |\ball Tvr| \,\le\,  2\bar\Delta \cdot (1+c)^2 \cdot g(2r)^2 
                \,\le\, \tfrac92\bar\Delta\cdot g(2r)^2.
} \hspace{16em}
\end{align*}

\item
If $\Gamma:= \smash{\displaystyle\sup_{m\ge n}}  \lceil\Delta(m)/\Delta(n) \rceil < \infty$, then $\bar\Delta<\infty$ and
\begin{align*}
\mathrlap{
     |\ball Tvr| \,\le\, 2(1+c) \cdot \big(\Gamma\cdot g(4r) - (\Gamma-1)\cdot g(2r)\big)
} \hspace{16em}
\\ 
\mathrlap{
 \phantom{|\ball Tvr| } 
 \,\le\, 3 \big(\Gamma\cdot g(4r) - (\Gamma-1)\cdot g(2r)\big)
} \hspace{16em}
\\ 
\mathrlap{
 \phantom{|\ball Tvr| } 
 %\,\le\, 3\max\{1,\Gamma\} \cdot g(4r).
 \,\le\, 3\Gamma \cdot g(4r).
} \hspace{16em}
\end{align*}
\end{myenumerate}
In particular, if $\Gamma<\infty$ then $T$ is of uniform volume growth $g$.
\end{theorem}

% \begin{theorem}\label{thm:main}
% For each $v\in V(T)$ and $r\ge 0$ holds
% %

% \begin{align*}
% \mathrlap{
%     |\ball Tvr| \,\ge\, (1-c) \cdot g(r/4) 
%                 \,\ge\, \tfrac12\, g(r/4)
% } \hspace{16em}
% \end{align*}
% %
% where $r/4$ in the argument of $g$ is rounded down to the largest integer.
 
% If $\bar\Delta:=\smash{\displaystyle\sup_n}\lceil\Delta(n)\rceil +1<\infty$, then $T$ is of maximum degree $\bar\Delta$ and
% %
% \begin{align*}
% \mathrlap{
%     |\ball Tvr| \,\le\,  2\bar\Delta \cdot (1+c)^2 \cdot g(2r)^2 
%                 \,\le\, \tfrac92\bar\Delta\cdot g(2r)^2.
% } \hspace{16em}
% \end{align*}

% If $\Gamma:= \smash{\displaystyle\sup_{m\ge n}}  \lceil\Delta(m)/\Delta(n) \rceil < \infty$, then $\bar\Delta<\infty$ and
% %
% \begin{align*}
% \mathrlap{
%      |\ball Tvr| \,\le\, 2(1+c) \cdot \big(\Gamma\cdot g(4r) - (\Gamma-1)\cdot g(2r)\big)
% } \hspace{16em}
% \\ 
% \mathrlap{
%  \phantom{|\ball Tvr| } 
%  \,\le\, 3 \big(\Gamma\cdot g(4r) - (\Gamma-1)\cdot g(2r)\big)
% } \hspace{16em}
% \\ 
% \mathrlap{
%  \phantom{|\ball Tvr| } 
%  \,\le\, 3\min\{1,\Gamma\} \cdot g(4r).
% } \hspace{16em}
% \end{align*}
%   % 
% In particular, if $\Gamma<\infty$ then $T$ is of uniform volume growth $g$.
% \end{theorem}

Note that the lower bound holds unconditionally. % (only subject to $\bar g$ being super-additive to ensure that $T$ exists, \cf\ \cref{rem:c}). %(only assuming $g$ is super-additive as discussed in \cref{rem:c}.
The conditions $\bar\Delta<\infty$~and~$\Gamma<\infty$ for the upper bounds are technical, but there are natural~criteria in terms of $g$ that imply both, such as being \emph{log-concave} (see \cref{res:conditions_in_g} below).

\begin{proof}[Proof of \cref{thm:main}]
%Fix $v\in V(T)$ and $r\ge 4$.
For a vertex $w\in V(T)$ we use the notation $\tau_n(w)$ to denote the unique copy $\prec_n T$ that contains $w$.
%
%Moreover, let us agree that whenever we write $T_n\subset T$, we refer to the canonical sub-graph $T_n$ in $T$ from the defining inclusion chain $T_0\subset T_1\subset T_2\subset\cdots$ in \cref{constr:tree}.

To prove \itm1 choose $n\in\NN_0$ with $2^{n+1}\le r\le 2^{n+2}$. 
%Let $x$ be the root of $T_n(v)$. 
Each vertex of $\tau_n(v)$ can be reached from $v$ in at most $2(2^n-1)$ steps: 
at most $2^n-1$ steps from $v$ to the~center of $\tau_n(v)$, and at most $2^n-1$ steps from the center to any other vertex of $\tau_n(v)$.\nlspace 
This yields
\begin{align*}
|\ball Tvr|
%|B_v^T\!(r)|
\ge |\ball Tv{2^{n+1}}|
%|B_v^T\!(2^{n+1})|
> |\ball Tv{2(2^n-1)}|
%|B_v^T\!(2(2^n-1))|
\ge |T_n| 
&
\ge (1-c)\cdot g(2^n) 
\\ &
\ge (1-c)\cdot g(r/4).
\end{align*}
Here and in the following the final form of the bound is obtained by using $c\le 1/2$.

% For \itm2 observe first that $T_{n-1}$ has $\delta_1\cdots\delta_{n-1}$ apocentric vertices, each of which is a leaf. 
% Hence, when forming $T_n$ by attaching $\delta_n$ new copies of $T_{n-1}$ to these,\nlspace and assuming a maximally uniform distribution of the attachment among them, we find that none of the formerly apocentric vertices will attain a degree above $\lceil\Delta(n)\rceil+1$ (but some will in fact attain this degree).
% Thus, if $\lceil\Delta(n)\rceil+1$ is bounded, then this gives the maximum degree of $T$.

To prove \itm2 choose $n\in\NN$ with $2^{n-1}\le r\le 2^n$.
Let $x$ be the center of $\tau_n(v)$, and, if it exists,
%(if $x$ is not the global root)
let $w$ be the unique neighbor of $x$ outside of $\tau_n(v)$.
%the parent of $x$ (\ie\ the unique neighbor of $x$ that lies closer to the global root).
%be the apocentric vertex of the central copy of $T_n$ in $T_{n+1}(v)$ to which $x$ is adjacent.
Which other copies $\prec_n T$ are reachable from $v$ within $2^n$ steps? The following list is exhaustive:
%The following list contains all copies $\tau\prec_n T$ that might be reached from $v$ in at most $2^n$ steps (some of them might actually not be reachable, but we are not missing any):
%
\begin{itemize}
    \item $\tau_n(v)$ and $\tau_n(w)$, % itself,
    \item a copy $\tau \prec_n T$ adjacent to an apocentric vertex of $\tau_n(v)$ \resp\ $\tau_n(w)$.
    %There are less than $|T_n|$ apocentric vertices and each one is attached to at most $\bar\Delta-1$ further copies of $T_n$.
    %\item $\tau_n(w)$,
    %\item a copy $\tau\prec_n T$ adjacent to an apocentric vertex of $\tau_n(w)$.
\end{itemize}
Both $\tau_n(v)$ and $\tau_n(w)$ have not more than $|T_n|$ apocentric vertices, each of which~is adjacent to at most $\smash{\bar\Delta-1}$ copies $\tau\prec_n T$.
In conclusion, 
%Thus, including $\tau_n(v)$ and~$\tau_n(w)$, we count that 
at most $2\bar\Delta\cdot|T_n|$ copies of $T_n$ are reachable from $v$ in $2^n$ steps, containing a total of at most $2\bar\Delta\cdot |T_n|^2$ vertices.
We therefore find
\begin{align*}
%|B_v^T\!(r)| 
|\ball Tvr|
\le |\ball Tv{2^n}|
%|B_v^T\!(2^n)| 
\le 2\bar\Delta\cdot|T_n|^2 
&\le 2\bar\Delta\cdot(1+c)^2\cdot g(2^n)^2 
\\&\le 2\bar\Delta \cdot (1+c)^2 \cdot g(2r)^2.    
\end{align*}

One might suggest a much better estimation for the number of copies $\tau\prec_n T$ reachable through apocentric vertices of $\tau_n(v)$ \resp\ $\tau_n(w)$: namely $\delta_{n+1}$.
However, this is only true if $\tau_n(v)$ is a central copy in $\tau_{n+1}(v)$.
In general, $\tau_n(v)$ can as well be apocentric in $\tau_{n+1}(v)$, or can even be apocentric in $\tau_m(v)$ for some very large $m > n$.
If so, then the number of copies $\prec_n T$ reachable through apocentric vertices of $\tau_n(v)$ is more plausibly related to $\delta_{m+1}$ than $\delta_{n+1}$, which is the reason for the very crude estimation above. %We make this precise below.
With a more~careful analysis we can prove the second upper bound.

To show \itm3 let $m \ge n$ be maximal so that $\tau_n(v)$ is apocentric in $\tau_m(v)$.
Such an $m$ might not exist (\eg\ when $T$ is an apocentric limit), in which case the apocentric vertices of $\tau_n(v)$ are leaves in $T$ and no other copies $\prec_n T$ are reachable through them.
If $m$ exists however, then $\tau_m(v)$~is necessarily central in $\tau_{m+1}(v)$. 
In particular, the $\delta_1\cdots \delta_m$ apocentric vertices of $\tau_m(v)$ are adja\-cent to $\delta_{m+1}$ apocentric copies $\prec_{m} \tau_{m+1}(v)$.
By the maximally uniformly distribution of adjacencies we can estimate how many of these copies are reachable through the $\delta_1\cdots\delta_n$ apocentric vertices of $\tau_n(v)$.
Following \cref{rem:max_uniform} this number is at most
\begin{align*}
\Big\lceil \frac{\delta_1\cdots\delta_n}{\delta_1\cdots\delta_m} \cdot \delta_{m+1}  \Big\rceil
&=
\Big\lceil \frac{\delta_{m+1}}{\delta_1\cdots\delta_m} \cdot \Big( \frac{\delta_{n+1}}{\delta_1\cdots\delta_n}\Big)^{-1} \cdot \delta_{n+1} \Big\rceil
\\&=\Big\lceil\frac{\Delta(m+1)}{\Delta(n+1)}\cdot\delta_{n+1}\Big\rceil
\le  \Gamma \cdot\delta_{n+1}.
\end{align*}
%
%copies of $T_n$ (inside of copies of $T_m$) adjacent to the apocentric vertices of $T_n(v)$.
This yields an improved upper bound on the number of copies $\tau\prec_n T$ adjacent to apo\-centric vertices of $\tau_n(v)$.
The same argument applies to $\tau_n(w)$.
Including $\tau_n(v)$ and $\tau_n(w)$ there are then at most $2\Gamma \cdot\delta_{n+1} + 2$ copies $\prec_n T$ reachable from $v$ in $2^n$ steps, and thus:
\begin{align*}
|\ball Tvr|
%|B_v^T\!(r)| 
\le |\ball Tv{2^n}|
%|B_v^T\!(2^n)| 
&\le (2\Gamma\cdot\delta_{n+1}+2)|T_n|
%\\ &\le 2\Gamma\cdot\delta_{n+1}\cdot|T_n| + 2|T_n|
\\ &\le 2\big(\Gamma\cdot (\delta_{n+1}+1)|T_n| - (\Gamma-1)\cdot|T_n|\big)
\\ &\le 2\big(\Gamma\cdot |T_{n+1}| - (\Gamma-1)\cdot|T_n|\big)
\\ &\le 2(c+1)\cdot\big(\Gamma\cdot g(2^{n+1}) - (\Gamma-1)\cdot g(2^n)\big).
\\ &\le 2(c+1)\cdot\big(\Gamma\cdot g(4r) - (\Gamma-1)\cdot g(2r)\big).
\end{align*}
%
%This establishes the improved upper bound.
\end{proof}

%One might wonder for which $g$ the technical condition $\Gamma<\infty$ in \cref{thm:main}~is satisfied.
%For the remainder of this section we make an effort to derive various~sufficient conditions.

It appears non-trivial to actually prescribe a natural growth rate $g$ for which $\Gamma$ diverges, and so we feel confident that \cref{thm:main} applies quite generally.
Yet,\nlspace for the remainder of this section we discuss criteria that are sufficient to imply  $\Gamma<\infty$.

\begin{proposition}\label{res:conditions_in_d}
$\Gamma<\infty$ holds in any of the following cases:
\begin{myenumerate}
    \item
    $\delta_{n+1}\le \delta_n^2$ eventually.
    % \item
    % $\delta_{n+1}\ge \delta_n^2$ and $\delta_{n+1}\le C\delta_n^2$ for some $C>1$ eventually.
    \item 
    the sequence of $\delta_n$ is bounded.
\end{myenumerate}
\begin{proof}
Note that $\Delta(n+1)=\delta_{n+1}/\delta_n^2 \cdot\Delta(n)$.
Assuming \itm1, $\Delta(n)$ will eventually~be non-increasing and so either $\Gamma\le1$ or the value $\Gamma$ is attained only on a finite initial segment, thus, is finite.

For part \itm2 observe $\Delta(m)/\Delta(n)=\delta_m / (\delta_n^2\cdot\delta_{n+1}\cdots\delta_{m-1})\le \displaystyle\smash{\max_i \delta_i}$.
\end{proof}
\end{proposition}

Condition \itm2 of \cref{res:conditions_in_d} holds, for example, when $g$ is a polynomial (see \cref{ex:polynomial}).
%Note that this is not quite the same as being of polynomial growth $g(r)\in\theta(r^\alpha)$, which is not sufficient.

% Part \itm1 of \cref{res:conditions_in_d} holds for all $g$ that are not growing to fast.
% Part \itm2 includes in particular all $g$ of polynomial growth (\cf\ \cref{ex:lattice_tight_upper_bound}.

% One natural condition that \emph{almost} makes the preconditions of \itm2 and \itm3~satisfied is requiring $g$ to be \emph{log-concave}.
% The relevant instance of this is
% %
% $$
% \tfrac13 \log g(2r) + \tfrac23 \log g(r/2) \le \log 
% g(r)
% %g\big(\overbrace{\tfrac13 (2r) + \tfrac23 (r/2)}^{=\, r}\big).
% $$
% %
% (note that indeed $\frac13(2r)+\frac23(r/2)=r$).
% Unfortunately, this is not quite sufficient, and we instead need an $\eps>1$ for which the following stronger equation holds
% %
% $$
% \tfrac13 \log g(2r) + \tfrac23 \log g(r/2) \le \eps \log 
% g(r),
% $$
% %
% at least eventually.

% \begin{proposition}
% If for some $\eps>1$ holds
% %
% $$
% \tfrac13 \log g(2r) + \tfrac23 \log g(r/2) \le \eps \log 
% $$
% %
% eventually (\ie\ for all sufficiently large $r\ge1$), then the conditions of \itm2 and \itm3 in \cref{thm:main} are satisfied.
% %
% \begin{proof}
% \TODO
% \end{proof}
% \end{proposition}

% In general, $g$ being log-concave is already a pretty good indicator that the bounds of \cref{thm:main} apply. Constructing counterexamples in a reasonable growth range is possible but needs some work.

% \begin{example}
% \TODO
% \end{example}

It remains to provide a criterion for $\Gamma<\infty$ in terms of $g$.
Below we show that~$g$ being \emph{log-concave} is sufficient, where log-concave means
$$\alpha\log g(x) + \beta \log g(y) \le \log g(\alpha x + \beta y)$$
for all $\alpha,\beta\in[0,1]$ with $\alpha+\beta=1$.
Intuitively, being log-concave is sufficient~because it prevents both super-exponential growth and strong oscillations in the growth rate. 

\begin{theorem}\label{res:conditions_in_g}
If $g$ is eventually log-concave, then $\Gamma<\infty$.
\begin{proof}
%The presented proof consists of a rather intuitive, and a more technical~part.
We start from a well-chosen instance of log-concavity: set
$$\alpha=1/3,\quad \beta=2/3, \quad x=2^{n+1},\quad y=2^{n-1}.$$
and verify $\alpha x+\beta y = 2^n$.
Since $g$ is eventually log-concave, for sufficiently large $n$ holds
$$\tfrac13\log g(2^{n+1}) + \tfrac23 \log g(2^{n-1}) \le \log g(2^n).$$
This further rearranges to
$$\log g(2^{n+1}) + 2\log g(2^{n-1}) \le 3\log g(2^n) \;\implies\; \frac{g(2^{n+1})}{g(2^n)} \le \Big(\frac{g(2^n)}{g(2^{n-1})}\Big)^2\!.$$
Recall that roughly $\delta_n\approx g(2^n)/g(2^{n-1})$, and so the right-most inequality resembles $\delta_{n+1}\le \delta_n^2$.
If this were exact then $\Gamma<\infty$ would already follow from \cref{res:conditions_in_d} \itm1.
However, the actual definition of $\delta_n$ from the proof of \cref{res:d_i} only yields
%The reality however is less nice: from the actual definition of the $\delta_n$
%\cref{res:d_i}
%follows %, and using the actual definition of $\delta_n$ from \cref{res:d_i} yields:
%
\begin{equation}\label{eq:d_d_2}
\delta_{n+1}
\le \Big\lfloor\frac{g(2^{n+1})}{g(2^n)}\Big\rfloor
\le \frac{g(2^{n+1})}{g(2^n)}
\le \Big(\frac{g(2^n)}{g(2^{n-1})}\Big)^2
\le \Big\lceil\frac{g(2^n)}{g(2^{n-1})}\Big\rceil^2
\le (\delta_n+2)^2.
\end{equation}
%
% which we can rearrange to
% %
% \begin{equation}\label{eq:nu}
% \nu(n)
% :=\frac{\delta_{n+1}}{\delta_n^2}
% \le \Big(1+\frac2{\delta_n}\Big)^2
% %\le 1+\frac4{\delta_n}+\frac4{\delta_n^2}
% . %\le 1 + \frac{2}{\delta_n-2}.
% \end{equation}
% %
% Recall that $\Delta(n+1)=\nu(n)\Delta(n)$.
% By \eqref{eq:nu} it is possible that $\Delta$ is increasing and so we can not put a bound on $\Delta(m)/\Delta(n)$ right away.

This turns out to be sufficient for proving $\Gamma<\infty$, though the argument becomes more technical.
Define $\gamma\:\NN\to\RR$ with
$$
\gamma(1):=25,\qquad \gamma(\delta):=\prod_{k=0}^{\infty} %\Big(1+\frac4{\delta^{2^k}}+\frac4{\delta^{2^{k+1}}}\Big)
\Big(1+\frac2{\delta^{2^k}}\Big)^{\!2}\!
,\;\;\,\text{whenever $\delta\ge 2$.}
$$
To keep this proof focused, most technicalities surrounding this function have been moved to \cref{sec:gamma_properties}.
This includes a proof of convergence of the infinite product (\cref{res:estimate_gamma}) as well as a proof of the following crucial property $(*)$ of $\gamma$: for~any two integers $\delta_1,\delta_2\ge 1$ with $\delta_1\le(\delta_2+2)^2$ holds $\delta_1\gamma(\delta_1)\le\delta_2^2\gamma(\delta_2)$ (\cref{res:gamma_property_d_d2}).

Lastly, we define $\pot(n):=\gamma(\delta_n)\Delta(n)$, intended to represent the \emph{potential} for $\Delta$ to increase, as we shall see that it is an upper bound on $\Delta(m)$ for \emph{all} $m\ge n$.
This follows from two observations.
Firstly, $\Delta(m)< \pot(m)$ since $\gamma(\delta)>1$. Secondly, $\pot(n)$ is decreasing in $n$: since $\delta_{n+1}\le(\delta_n+2)^2$ by \eqref{eq:d_d_2} (if $n$ is sufficiently large), we can apply property $(*)$ to find
$$\pot(n+1) = \gamma(\delta_{n+1})\Delta(n+1) = \gamma(\delta_{n+1})\frac{\delta_{n+1}}{\delta_n^2}\Delta(n) \overset{\smash{(*)}}\le \gamma(\delta_n)\Delta(n)=\pot(n).$$
To summarize, for all sufficiently large $m\ge n$ holds
\begin{equation*}\label{eq:goal}
\Delta(m)<\pot(m)
%\overset{\mathclap{(*)}}
\le\pot(n)=\Delta(n)\gamma(\delta_n) \,\implies\, \frac{\Delta(m)}{\Delta(n)}\le\gamma(\delta_n)\le 25,
\end{equation*}
where we used that $\gamma(\delta)$ is decreasing in $\delta$ (\cref{res:gamma_decreasing}) and so attains its maximum at $\gamma(1)=25$. This proves that $\Gamma$ is finite.
\end{proof}
\end{theorem}

As a corollary of \cref{thm:main} and \cref{res:conditions_in_g} we have proven

\begin{theoremX}{\ref{main}}
    If \gFun\ is super-additive and (eventually) log-concave, then there exists a tree $T$ of uniform growth $g$. 
\end{theoremX}

\iftrue % remark for non-central inclusion chains

\section{Unimodular random trees}
\label{sec:unimodular}

In the remainder of the article we apply the theory of \emph{Benjamini-Schramm limits} to the sequence $T_n$ (\cref{constr:T_n}); we obtain \emph{unimodular random rooted trees of uniform intermediate growth}, answering a question by Itai Benjamini.
We investigate the structure of the limit graphs and provide an instance that is supported on a single deterministic tree.

A \emph{rooted graph} is a pair of the form $(G,o)$, where $G$ is a graph and $o\in V(G)$.\nlspace 
For a definition of \emph{random rooted graphs} and their basic properties we follow \cite{benjamini2011recurrence}: firstly, there is a natural topology on the set of rooted graphs -- the \emph{local topology} -- induced by the metric 
\begin{align*}
\dist\!\big((G,o),(G',o')\big):=2^{-R} \quad &\text{if } \ball Gor\cong \ball{G'}{o'}r \text{ for all } 0\leq r\leq R
\\
&\text{and } \ball Go{R+1}\ncong \ball{G'}{o'}{R+1},
\end{align*}
%
% \[d((G,o),(G',o'))=2^{-k} \hspace{1mm}\text{if}\hspace{1mm} \ball Gor\cong \bal{G'}{o'}r \hspace{1mm} \text{for}\hspace{1mm} 0\leq r\leq k\]
% %
% \[\text{and}\hspace{1mm} \ball Go{k+1}\ncong \ball{G'}{o'}{k+1},\]
%
where it is understood that $\ball{G}{o}{r}$ is rooted at $o$ and that isomorphisms between rooted graphs preserve roots.

A random rooted graph $(G,o)$ is a Borel probability measure (for the local topology) on the set of locally finite, connected rooted graphs. In particular, it is \emph{deterministic} if it is the Dirac measure over a single rooted tree.
We call $(G,o)$ \emph{finite} if the set of infinite rooted graphs has $(G,o)$-measure zero. 
If in addition the conditional distribution of the root in $(G,o)$ over each finite graph is uniform, then $(G,o)$ is called \emph{unbiased}. 
%If the conditional distribution of the root of a finite random rooted graph $(G,o)$ over each of the underlying graphs in its support is uniform, then $(G,o)$ is called \emph{unbiased}. 

Given a sequence $(G_n,o_n)$ of unbiased random rooted graphs, a random rooted graph $(G,o)$ is said to be the \emph{Benjamini-Schramm limit} of $(G_n,o_n)$ if for every rooted graph $(H,\omega)$ and natural number $r\ge 0$ we have 
\[\lim_{\mathclap{n\to \infty}}P\big(\ball{G_n\!}{o_n}r\cong (H,\omega)\big)=P\big(\ball{G}{o}r\cong (H,\omega)\big).\] 
Note that, if it exists, $(G,o)$ is the unique limit. %\cite{benjamini2011recurrence}.
If a random rooted graph~is~the~Ben\-jamini-Schramm limit of some sequence, we call it \emph{sofic}. 
%By a standard result,
%\cite{benjamini2011recurrence}, 
One can show that the set of graphs of maximum degree $\le \Delta$ is compact in the local topology, and thus, a sequence $(G_n,o_n)$ of uniformly bounded degree always has a convergent~subsequence.

We say that a random rooted graph $(G,o)$ is of \emph{uniform growth} $g$ if there are~constants $c_1, C_1, c_2, C_2>0$ such that $(G,o)$ is \as\ of uniform growth $g$ \wrt\ these~constants (as in \eqref{eq:uniform_growth}). 
For unimodular graphs (in particular, for all graphs that~we care about below), this is equivalent to the existence of constants $c_1$, $C_1$, $c_2$, $C_2>0$ so that \as
$$C_1\cdot g(c_1r)\leq |B_G(o,r)|\leq C_2\cdot g(c_2r),\quad\text{for all $r\ge 0$}.$$
%
%That is, almost all the vertices (measured by the distribution of $o$) exhibit a common uniform growth. 
Indeed, the forward implication is obvious.  
For the backward implication, note~that in a unimodular graph, by \cite[Proposition 11]{curien},  
%\todo[color=gray]{\msays\ This needs a source or explanation. Below we use something similar but we include a source.}
each rooted $r$-ball is isomorphic to $B_G(o,r)$~with~positive \mbox{probability}.

%The notion of uniform growth is well-defined in random rooted graphs. We say that a random rooted graph $(G,o)$ has uniform growth $g(r)$ if there exists a function $g:\RR_{\ge 0}\rightarrow \RR_{\ge 0}$ and positive constants $c_1, C_1, c_2, C_2$ such that for every $r\ge 0$ \as\ $C_1\cdot g(c_1r)\leq |B_G(o,r)|\leq C_2\cdot g(c_2r)$, \ie\ if almost all the vertices (measured by the distribution of $o$) exhibit a common uniform growth. 

The original formulation of Benjamini's question asks for \emph{unimodular} trees~of intermediate growth.
Unimodularity generalizes the concept of a uniformly chosen root to graphs that are not necessarily finite: a random rooted graph $(G,o)$~is~unimodular if it obeys the \emph{mass transport principle}, \ie
\[\mathbb E\Big[\sum_{\mathclap{x\in V(G)}}f(G,o,x)\Big]=\mathbb E\Big[\sum_{\mathclap{x\in V(G)}}f(G,x,o)\Big]\]
for every \emph{transport function} $f$, which, for our purpose, are sufficiently defined~as~non-negative real Borel functions over doubly-pointed graphs (for a precise definition~we direct the reader to \cite{aldouslyons}).
%. For our purpose the following definition will be sufficient: a transport function is a Borel function over doubly-pointed graphs that~outputs non-negative real numbers and simulates mass transport between vertices. 
The function $f$ simulates~mass transport between vertices, and the mass transport principle states, roughly, that the root $o$ sends, on average, as much mass to other vertices as it receives from~them.
Unimodular~graphs~are~significant in the theory of random graphs and encompass various important graph classes, most notably, all~sofic graphs.
The famous \emph{Aldous-Lyons conjecture} states the converse, that all unimodular random graphs are sofic \cite{aldouslyons}. 

In \cref{sec:main}, given a suitable function \gFun, we constructed a rooted tree $(T,x^*)$ of uniform growth $g$ ($x^*$\! being the vertex of $T_0\subset T$). 
Can~we turn this into a unimodular tree of uniform growth $g$?
Naturally, we can interpret $(T,x^*)$ as the random rooted graph \as\ being this particular rooted tree.
%Then $(T,x^*)$ is almost never unimodular:
It is known that for unimodular graphs the conditional distribution~of the root over each underlying graph has positive probability on each vertex orbit \cite[Proposition 12]{curien}.
Thus, if $T$ is not vertex-transitive, this interpretation of $(T,x^*)$ as a random rooted graph cannot be unimodular. % and unimodular trees of uniform growth $g$ need to be costructed differently.

% In fact, suppose $T$ has a vertex $v\in N_T(x^*)$ of a different degree than $x^*$. Set
% %
% $$
% f(G,x,y):=\begin{cases}
% 1 & \text{if $G=T$ and $x\sim y$} \\
% 0 & \text{otherwise}
% \end{cases},
% $$
% %
% then
% %
% \[
% \mathbb E\Big[\sum_{\mathclap{x\in V(T)}}f(T,x^*,x)\Big]
% = \deg(x^*)\not=\deg(v) =
% \mathbb E\Big[\sum_{\mathclap{x\in V(T)}}f(T,x,x^*)\Big]
% \]

% \msays{Work in progress ...}

% Indeed, suppose that we construct $(T,x^*)$ as a centric limit, and that $|\frac{\delta_n}{\delta_1...\delta_{n-1}}-(\delta_1-1)|\geq 1$ for every $n\geq 2$. Then $x^*$ is the only vertex with degree $\delta_1$. Consider the transport function $f(G,x,y)$ which equals $1$ if $G=T$, $x\sim y$ and $d_T(x,x^*)<d_T(y,x^*)$, and $0$ otherwise. Then $\mathbb{E}[\Sigma_{x\in V(T)}f(T,x^*,x)]=\delta_1\neq 0=\mathbb{E}[\Sigma_{x\in V(T)}f(T,x,x^*)]$.

That being said, unimodular trees with uniform growth $g$ can be obtained from the defining sequence $T_n, n\ge 0$ (\cref{constr:T_n}).
Let $(T_n, o_n)$ denote the unbiased random rooted graph that is \as\ isomorphic to $T_n$.

\begin{proposition}\label{converge?}
    Let \gFun\ be super-additive and log-concave. Then~the sequence $(T_n,o_n)$ has a subsequence that converges in the Benjamini-Schramm sense to a unimodular random rooted tree of uniform growth $g$. %In particular, there exists a unimodular tree with intermediate uniform growth.
\end{proposition}

\begin{proof}
%Consider a sequence $(T_n,o_n)$ of unbiased random rooted trees, where each $T_n$ is \as\ as in \cref{constr:T_n}. 
By \cref{res:conditions_in_g} the sequence $(T_n,o_n)$ has uniformly bounded degree, and, since the set of graphs with uniformly bounded degree is compact in the local~to\-pology, has a subsequence $(T_{k_n},o_{k_n})$ Benjamini-Schramm convergent to a random rooted tree $(\mathcal{T},\omega)$. 
By the definition of Benjamini-Schramm limit, a ball $\ball{\mathcal T}{\omega}{r}$ %in~$(\mathcal{T},\omega)$ centered at $\omega$ 
is \as\ isomorphic to a ball in $T_{k_n}$ for $k_n$ large enough. 
However, every ball of radius $r\leq 2^{k_n-1}$ in $T_{k_n}$ is isomorphic to a ball in some limit tree $T$ (\cf\  \cref{constr:T}) and by \cref{thm:main} therefore satisfies the bounds
$$C_1\cdot g(c_1r)\leq |\ball{\mathcal T}{\omega}{r}|\leq C_2\cdot g(c_2r),$$
with constants that only depend on the sequence $\delta_1,\delta_2,\delta_3,...$.
%
%Every ball of $(\mathcal{T},\omega)$ is contained in a ball rooted at  $\omega$, and is thus \as\ a ball of $\mathcal{T}_n$. 

In conclusion, $(\mathcal{T},\omega)$ exhibits uniform growth of order $g$. 
Since $(\mathcal{T},\omega)$ is sofic (\ie\ it is the Benjamini-Schramm limit of unbiased graphs), it is also unimodular.
\end{proof}

At this point we have proven

\begin{theoremX}{\ref{unimain}}
    If \gFun\ is super-additive and (eventually) log-concave, then there exists a unimodular random rooted tree $(\mathcal T,\omega)$ of uniform growth $g$. 
\end{theoremX}

We can be more precise about the convergence of the sequence $T_n$ as it is actually not necessary to restrict to a subsequence: % and eschew taking a subsequence:

\begin{proposition}\label{converge}
    The sequence $T_n,n\ge 0$ is Benjamini-Schramm convergent.
\end{proposition}

\begin{proof}
    Fix $r\ge 1$ and a graph $H$.
    Call a vertex $v\in V(T_n)$ \emph{good} if $\ball{T_n} v r\cong H$.
    Let $P_n$ be the probability that a uniformly chosen vertex of $T_n$ is good.
    %We show $|P_{n+k}-P_n|\le C_r 2^{-n}$ for all $k\ge 0$ and a constant $C_r\ge 1$ that only depends on $r$.
    %In other words, $P_n$ is a Cauchy sequence, hence convergent, and $T_n$ is BS-convergent.
    We show that $P_n$ is a Cauchy sequence, hence convergent.

    Let $s_n$ be the number of good vertices in $T_n$.
    Since $T_n$ consists of $\delta_n+1$ copies of $T_{n-1}$, one might expect $s_n\approx (\delta_n+1)s_{n-1}$.
    However, besides the copies, $T_n$ also consists of $\delta_n$ new edges, and so for $\tau\prec_{n-1} T_n$ and $v\in V(\tau)$, the $r$-ball around $v$ in $\tau\cong T_{n-1}$ might look different from the $r$-ball around $v$ in $T_n$.
    Still, this can only happen if $v$ is $r$-close to one of the $2\delta_n$ ends of the new edges.
    Let $C_r$ be an upper bound on the size of an $r$-ball in $T$ (which exists since $T$ is of bounded degree, \cf\ \cref{rem:dbar}). We can then conclude
    $$|s_n-(\delta_n+1)s_{n-1}|\le 2C_r\delta_n.$$
    Since $P_n=s_n/|T_n|$, division by $|T_n|=(\delta_n+1)|T_{n-1}|$ yields
    \begin{align*}
        &|P_n-P_{n-1}|\le \frac{\delta_n}{\delta_n+1}\frac{2C_r}{|T_{n-1}|}\le \frac{2C_r}{|T_{n-1}|}.
    \end{align*}
    It follows for all $k\ge 1$
    \begin{align*}
    |P_{n+k}-P_n| 
    &\le |P_{n+k}-P_{n+k-1}|+\cdots+|P_{n+1}-P_n|
    \\ &\le 2C_r\Big(\frac1{|T_{n+k-1}|}+\cdots+\frac1{|T_{n}|}\Big)
    \le 2C_r\sum_{i=n}^{\infty} \frac1{|T_i|}.
    \end{align*}
    Since $\delta_n\ge 1$ and $|T_n|=(\delta_1+1)\cdots(\delta_n+1)$ we have $|T_n|\ge 2^n$. It therefore follows
    \begin{align*}
    |P_{n+k}-P_n| \le \frac{4C_r}{2^n},
    \end{align*}
    and $P_n$ is indeed a Cauchy sequence.
\end{proof}

In the following, let $(\mathcal T,\omega)$ denote the Benjamini-Schramm limit of the sequence $T_0, T_1, T_2,...$, where $T_n$ is short for the unbiased random rooted tree $(T_n,o_n)$.

While \cref{converge?} and \cref{converge} establish that $(\mathcal T,\omega)$ exists, they are not particularly enlightening with regards to the structure of the obtained limit. 
We now prove that $(\mathcal T,\omega)$ is \as\ a limit tree in the sense of \cref{constr:T}.
We shall also find a qualitative differences between limit trees obtained from a sequence $\delta_1,\delta_2,\delta_3,...$ with $\sum_n\delta_n^{-1}<\infty$ versus $\sum_n\delta_n^{-1}=\infty$.
The former sequences we shall call \emph{fast-growing}, and the latter \emph{slow-growing}.
The main structure theorem reads as follows:

\begin{main}[Structure Theorem]\label{structure}
Let $\delta_1,\delta_2,\delta_3,...\ge 2$ be a sequence of integers.
\begin{myenumerate}
    \item
    If it is fast-growing, then $(\mathcal{T},\omega)$ is \as\ an apocentric limit.% In particular, it is one-ended.
    \item
    If it is slow-growing, then $(\mathcal{T},\omega)$ is \as\ a mixed limit and the probability for being isomorphic to any particular deterministic tree is zero. % that any specific tree is its underlying tree is 0.    
\end{myenumerate}
In both cases $(\mathcal T,\omega)$ is \as\ 1-ended.
\end{main}

\begin{remark}
If $\delta_n\in\Omega(n^\alpha)$ for some $\alpha>1$ then it is fast-growing, whereas if $\delta_n\in O(n)$ then it is slow-growing.
In terms of $g$, the threshold of \cref{structure} can be located roughly at $g(r)=r^{\alpha \log\log r}$ between $\alpha=1$ and $\alpha>1$.
In fact, note that
$$g(2^n)=(2^n)^{\alpha\log\log 2^n} = n^{\alpha n}$$
and according to \cref{res:d_i}
$$
\delta_n\sim\frac{g(2^{n+1})}{g(2^n)}
=(n+1)^\alpha \cdot \Big(1+\frac1n\Big)^{\alpha n} \sim (n+1)^\alpha e^\alpha \sim n^\alpha.
$$
\end{remark}

To establish \cref{structure} we prove two auxiliary lemmas.

\begin{lemma} \label{excellent}
    Suppose that there exists a sequence $k_n$ such that, for every $n$, $\omega$ \as\ lies in a tree $\tau_n\cong T_{k_n}$, the center of which is incident to an edge that separates it from $\mathcal{T}\setminus \tau_n$. Then $\mathcal{T}$ is \as\ 1-ended and a limit tree in the sense of \cref{constr:T}.
\end{lemma}

\begin{proof}
    We begin by noting that for each $n\in \mathbb{N}$, $\tau_n$ is unique (and may therefore be unambiguously denoted $\tau_n(\omega)$). 
    Indeed, if there exists another tree $\tau_n'$ with the same properties, then $\omega\in\tau_n\cap\tau_n'$. 
    Since the centre of $\tau_n$ is incident to an edge that separates it from $\mathcal{T}\setminus \tau_n$, no part of $\tau_n'\setminus \tau_n$ may be in the same side of that edge as $\omega$. Therefore either $\tau_n'\setminus \tau_n=\emptyset\Rightarrow \tau_n'\subseteq \tau_n$, or else $\tau_n\subseteq \tau_n'$. Since the two trees are isomorphic, we obtain equality. A similar argument yields $\tau_n(\omega)\subseteq \tau_{n+1}(\omega)$ for every $n\in\mathbb{N}$.  
    
    Explicitly, the corresponding inclusion chain limit $(\mathbf{T},\mathbf{o})$ is simply the limit of the finite inclusion chains of the copies $(\tau_n,\omega)$. This exists, since these chains are compatible (i.e. each is a restriction of the next). It is straightforward to see that $(\mathcal{T},\omega)\cong(\mathbf{T},\mathbf{o})$ by comparing their rooted balls of equal radius. 

    One ray $R_1$ of $(\mathbf{T},\mathbf{o})$ is obtained by concatenating the paths $P_n$ from $x_n$ to $x_{n+1}$, where $x_0=\mathbf{o}$ and $x_n$ is the center of $\tau_n(\mathbf{o})$ for $n>0$. Any other ray $R_2$ has to begin inside a ball centered at $\mathbf{o}$, hence inside some $\tau_n(\mathbf{o})$, which implies, since the latter is separated from $\mathcal{T}$ by an edge incident to its center, that $R_2$ must contain the center of $\tau_n(\mathbf{o})$, thus intersecting $R_1$. We deduce that $\mathbf{T}$ is 1-ended.
\end{proof}

\begin{lemma} \label{slow}
    Let $\delta_1,\delta_2,\delta_3,...$ be a slow-growing sequence with $\delta_n\ge 2$ eventually.
    For a $\nu>1/2$ let $\delta_{n_k}$ be the subsequence of elements satisfying $\delta_{n_k}\le \nu \delta_{n_k-1}\delta_{n_k-2}$. 
    Then $\delta_{n_k}$ is itself slow-growing.
\end{lemma}

\begin{proof}
%\todo{@Martin: this proof needs some fixing.}
    Set $N_k:=n_{k+1}-n_{k}-1$. Then $\delta_{n_k+i}>\nu\delta_{n_k+i-1}\delta_{n_k+i-2}$ for \mbox{$i\in\{1,...,N_k\}$}.
    Since $\delta_n\ge 2$ eventually, there is a $k_0$ so that $\delta_{n_k-1}\ge 2$ for all $k\ge k_0$.
    We then first verify the following: for $k\ge k_0$ holds
    $$(*)\;\;\delta_{n_k+i}\ge \tfrac1\nu (\nu \delta_{n_k})^{F_{i+1}}\quad\text{for $i\in\{0,...,N_n\}$},$$
    where $F_i$ is the $i$-th Fibonacci number starting with $F_1=F_2=1$.
    We verify this inductively:
    \begin{align*}
        \delta_{n_k+0} &= \tfrac1\nu (\nu\delta_{n_k})^{F_1},
        \\
        \delta_{n_k+1} &> \nu \delta_{n_k}\delta_{n_k-1}>\tfrac12\cdot\delta_{n_k}\cdot 2= \delta_{n_k} = \tfrac1\nu (\nu \delta_{n_k})^{F_2},        
        \\
        \delta_{n_k+i} &> \nu\delta_{n_k+i-1}\delta_{n_k+i-2}
        \\ &\ge \nu\cdot\tfrac1\nu(\nu\delta_{n_k})^{F_{i}}\cdot\tfrac1\nu(\nu\delta_{n_k})^{F_{i-1}}
        = \tfrac1\nu(\nu\delta_{n_k})^{F_i+F_{i-1}} = \tfrac1\nu(\nu\delta_{n_k})^{F_{i+1}}
    \end{align*}
    Since $\delta_n$ is slow growing, we have
    \begin{align*}
       \infty= \sum_{\mathclap{n\ge n_{k_0}}}\delta_n^{-1} &= \sum_{k\ge k_0} \sum_{i=0}^{N_k} \delta_{n_k+i}^{-1} 
        %\\&
        \overset{(*)}< \nu \sum_{k\ge k_0} \sum_{i=0}^{N_k} (\nu \delta_{n_k})^{-F_{i+1}}
    %\le \nu \sum_{k} \sum_{i=0}^{\infty} (\nu \delta_{n_k})^{-F_{i+1}}
        \\&= \nu \sum_{k\ge k_0}(\nu \delta_{n_k})^{-1}\sum_{i=0}^{N_k} (\nu \delta_{n_k})^{-F_{i+1}+1}
        \\ &\le \nu \sum_{k\ge k_0} (\nu \delta_{n_k})^{-1}\underbrace{\sum_{i=0}^{\infty} (2\nu)^{-(F_{i+1}-1)}}_{=:\, K>0}
        = K \sum_{k\ge k_0} \delta_{n_k}^{-1},
    \end{align*}
    where the infinite series converges because $2\nu>2\cdot \tfrac12=1$ and the exponents $F_{i+1}-1$ grow exponentially fast.
    Hence, $\delta_{n_k}$ is slow-growing.
    %
    % Since $\nu\delta_{k_n}> \tfrac12\cdot 2=1$, the following series converge:
    % %Recall that $F_i\le \varphi^i/\sqrt 5\le i\varphi/\sqrt 5$ if $i\ge 5$. Therefore
    % %
    % \begin{align*}
    %     \sum_{i=0}^{k_n} (\nu \delta_{k_n})^{-F_{i+1}} 
    %     %\le \sum_{i=0}^{\infty} (\nu \delta_{k_n})^{-F_{i+1}} 
    %     &= (\nu \delta_{k_n})^{-1}\sum_{i=0}^{k_n} (\nu \delta_{k_n})^{-F_{i+1}+1}
    %     \le (\nu \delta_{k_n})^{-1}\underbrace{\sum_{i=0}^{\infty} (2\nu)^{-F_{i+1}+1}}_{=:\, K}=\tfrac K\nu \delta_{k_n}^{-1}.
    % \end{align*}
    % \begin{align*}
    %     \sum_{i=0}^{\infty} (\nu \delta_{k_n})^{-F_{i+1}} 
    %     &= (\nu \delta_{k_n})^{-1}+(\nu \delta_{k_n})^{-1}+(\nu \delta_{k_n})^{-2}+(\nu \delta_{k_n})^{-3}+\sum_{i=5}^{\infty} (\nu \delta_{k_n})^{F_{i}}
    %     \\&\le 2(\nu \delta_{k_n})^{-1}+(\nu \delta_{k_n})^{-2}+(\nu \delta_{k_n})^{-3}+\sum_{i=5}^{\infty} ((\nu \delta_{k_n})^{-\varphi/\sqrt{5}})^i
    %     \\&= (\nu \delta_{k_n})^{-1}\Big(2+(\nu \delta_{k_n})^{-1}+(\nu \delta_{k_n})^{-2}+\frac{(\nu \delta_{k_n})^{-5\varphi/\sqrt{5}-1}}{1-(\nu \delta_{k_n})^{-\varphi/\sqrt{5}}}\Big) 
    %     \\&\le (\nu \delta_{k_n})^{-1}\Big(\underbrace{2+(2\nu)^{-1}+(2\nu)^{-2}+\frac{(2\nu)^{-5\varphi/\sqrt{5}-1}}{1-(2\nu)^{-\varphi/\sqrt{5}}}}_{=:\,K}\Big)=\tfrac K\nu \delta_{k_n}^{-1}
    % \end{align*}
    %
    % So we have
    % %
    % $$\sum_n\delta_n^{-1} \le K \sum_k \delta_{k_n}^{-1},$$
    % %
    % and since $\delta_n$ is slow-growing, so is $\delta_{k_n}$.
\end{proof}

We are now ready to prove \cref{structure}.

\begin{proof}[Proof of \cref{structure}]%\quad 
Define $p_n:=\delta_n/(\delta_n+1)$ and note that this is precisely the probability that a vertex $v$ chosen uniformly from $T_N$ (for an arbitrary fixed $N\geq n$) satisfies  $\tau_{n-1}(v)\prec_{n-1} \tau_n(v)\prec_n T_N$ with the first inclusion being apocentric. 

We observe that the product $\prod_n p_n$ is positive if and only if $\delta_1,\delta_2,\delta_3,...$ is a fast-growing sequence: we know that $\prod_n p_n$ is positive if and only if its inverse $(\prod_n p_n)^{-1}=\prod_n (1+\delta_n^{-1})$ converges. As is well-known, the latter converges if and only if $\sum_n \delta_n^{-1}$ converges.

We start with the proof of \itm1. For a fast-growing sequence and any particular $k\in\mathbb{N}$, the tail product $p^{(k)}:=\prod_{n=k+1}^{\infty}p_n$ is positive. Additionally, note that $p^{(k)}$ is the limit of the probability that $o_n$ (uniformly chosen from $T_n$) lies in an apocentric copy $\tau\prec_k T_n$ as $n\to\infty$. 
    
Consider the event $A_{k,r}$ that the ball $B_{\mathcal{T}}(\omega,r)$ is isomorphic to a ball centered at a vertex of an apocentric copy $\tau\prec_k T_n$ (for $n=n(r)=2^k+r$, \ie\ large enough that the center of $T_n$ is not contained in the ball of radius $r$ centered at the center of $\tau$).
Note that for every $r\geq 0$, $\P(A_{k,r})\geq p^{(k)}$. 
Let $A_k:=\bigcap_{r\geq 0}A_{k,r}$. That is, $A_k$ is the event that for every $r\geq 0$ the ball $B_{\mathcal{T}}(\omega,r)$ is isomorphic to a ball centered at an apocentric $\tau\prec_k T_n$. Since $A_{k,r}$ is decreasing in $r$, we have 
$$\P(A_k)=\lim_{r\rightarrow \infty}\P(A_{k,r})\geq p^{(k)}.$$
Since $A_k$ is increasing in $k$, 
$$\P\Big(\!\bigcup_{k\geq 0}\!A_k\Big)=\lim_{k\rightarrow \infty}\P(A_k)\geq \lim_{k\rightarrow \infty}p^{(k)}=1.$$
Therefore, there exists \as\ $k\in \NN_0$ such that the event $A_k$ occurs. 
That means, each ball around $\omega$ shows it lying in an apocentric copy of $T_n$ for every sufficiently large $n$. By \cref{excellent}, this implies that $(\mathcal{T},\omega)$ is \as\ an apocentric limit, and in particular 1-ended.   

%Let $v\in V(\mathcal{T})$ and let $m\geq k$ be such that a copy $\tau\prec_m\mathcal{T}$ contains both $\omega$ and $v$. Take $r$ large enough so that $N(\tau)\subset B_{\mathcal{T}}(\omega,r)$ and recall that (since $A_k$ is increasing and $m\geq k$) $B_{\mathcal{T}}(\omega,r)$ is isomorphic to a ball of a tree $T_n$ centered at a vertex of an apocentric $\tau'\prec_m T_n$. Since the apocentric vertices of $\tau'$ are leaves in $T_n$, the apocentric vertices of $\tau$ are leaves in $\mathcal{T}$, so the only vertex of $\tau$ that has a neighbor outside of $\tau$ is its center. Only one of the connected components of $\mathcal{T}\setminus v$ contains the center of $\tau$, so the rest are subtrees of $\tau$, therefore finite. Hence, $\mathcal{T}$ is 1-ended.

Subsequently, we proceed with the proof of \itm2. For a slow-growing sequence, let $p_{(k)}:=\prod_{n=1}^kp_n$ and note that $p_{(k)}$ is the probability that a uniformly chosen vertex $v$ of a tree $T_n$, $n\geq k$ is apocentric in a copy $\tau\prec_k T_n$. 
Such a vertex will be called a $k$\emph{-apocentric vertex}, and any ball centered at it will be called a $k$\emph{-apocentric ball}. 
Let $B_{k,r}$ be the event that the ball $B_{\mathcal{T}}(\omega,r)$ is isomorphic to a $k$-apocentric ball of $T_n$ for $n$ large enough. 
Note that, for $r$ (and $n$) large enough, the only vertices in $T_n$ that are the centers of balls isomorphic to $k$-apocentric balls are indeed $k$-apocentric vertices, hence $\P(B_{k,r})=p_{(k)}$. Also, observe that $B_{k,r}$ is a decreasing sequence in $r$. Let $B_k:=\bigcap_{r\geq 0}B_{k,r}$. That is, $B_k$ is the event that for every $r\geq 0$ the ball $B_{\mathcal{T}}(\omega,r)$ is isomorphic to a $k$-apocentric ball in $T_n$ for $n$ large enough. We deduce that 
$$\P(B_k)=\lim_{r\rightarrow \infty}\P(B_{k,r})=  p_{(k)}.$$ 
Since $B_k$ is also a decreasing sequence in $k$, 
$$\P\Big(\!\bigcap_{k\geq 0}\!B_k\Big)=\lim_{k\rightarrow \infty}\P(B_k)=\lim_{k\rightarrow \infty}p_{(k)}=0.$$ 
Therefore $B_{\mathcal{T}}(\omega, r)$ is \as\ not isomorphic to a $k$-apocentric ball for some $k, r\in \NN$, \ie\ $\tau_{k_0}(\omega)\prec_{k_0}\tau_{k_0+1}'(\omega)\prec_{k_0+1}\mathcal{T}$ and $\tau$ is central in $\tau'$ for some $k_0<k$. Since for any $k_n\in\mathbb{N}$ 
$$\P\Big(\!\bigcap_{k\geq k_n}\!B_k\Big)=0,$$ 
we may repeat the above argument to iteratively obtain a sequence of copies $\tau_{k_n}(\omega)\prec_{k_n}\tau_{k_n+1}(\omega)\prec_{k_n+1}\mathcal{T}$ with $\tau_{k_n}$ central in $\tau_{k_n+1}$.
The proof that $\omega$ lies \as\ in a sequence of copies $\tau_{l_n}(\omega)\prec_{l_n}\tau_{l_n+1}(\omega)\prec_{l_n+1}\mathcal{T}$ with $\tau_{l_n}$ apocentric in $\tau_{l_n+1}$ is completely analogous.

Note that, since we have a slow-growing sequence, there are infinitely many indices $k_n$ for which $\delta_{k_n+2}\le \frac23 \delta_{k_n+1}\delta_{k_n}$, 
%\todo{Can't we have $\delta_{k_n+2}\ge \frac23 \delta_{k_n+1}\delta_{k_n}$ for all $n$ if the $k_n$ are just sufficiently sparse?}
which yield sub-trees of $\mathcal{T}$ that are isomorphic to $T_{k_n}$ and separated from the rest of $\mathcal{T}$ by an edge incident to their center. Moreover, there is a subsequence of these indices $k_n$ such that $\omega$ lies in these trees $T_{k_n}$. Indeed, the probability of the complement is at most $\sum_{k=0}^{\infty}\prod_{i=k}^{\infty}q_{k_n}$, where $q_{k_n}:=\frac{\delta_{k_n}+\frac{7}{9}}{\delta_{k_n}+1}$. By \cref{slow} this equals zero. By \cref{excellent} we derive that $(\mathcal{T},\omega)$ is a 1-ended mixed limit.

Finally, it remains to show that $(\mathcal{T},\omega)$ has zero probability of being any particular deterministic tree. Note that, if $\delta_n>1$ eventually, then each inclusion chain yields a different sequence of balls $B_{\mathcal{T}}(\omega,r), r\in\{1,2,3,...\}$, \ie\ a different rooted limit tree. Also, for a slow-growing sequence in particular, $\prod_{n=1}^{\infty}\mathcal{P}_n=0$, where each $\mathcal{P}_n$ can be equal to either $p_n$ or $1-p_n$. That is to say, each particular inclusion chain, and therefore from the previous observation each rooted limit tree of \cref{constr:T}, is sampled from our mixed limit tree $(\mathcal{T},\omega)$ with probability zero. Since each such tree has only countably many vertices, any particular tree has probability zero to be the underlying tree of $(\mathcal{T},\omega)$. 
\end{proof}

Due to the semblance of apocentric limits to the classic canopy tree, we also~use the term \emph{generalized canopy tree} to refer to the limit tree $(\mathcal{T},\omega)$ induced~by~a~fast-growing sequence.

The significance of the distinction worked out in \cref{structure} becomes more appa\-rent with an example: generalized canopy trees (even of intermediate growth) can~be \as\ isomorphic to a single deterministic tree.
In essence, such limits can be seen as deter\-ministic trees with a randomly chosen root.
This is not possible for limit trees induced by slow-growing sequences (ex\-cept for trivial cases such as $\delta_n = 1$).

\begin{example}\label{ex:deterministic}
Define recursively $\delta_1:=1,\delta_2:=2$ and $\delta_{n+1}:=\delta_n\delta_{n-1}$.
In explicit form this reads
$$\delta_n=2^{F_n}, \quad\text{for all $n\ge 0$},$$
where $F_n$ denotes the $n$-th Fibonacci number starting with $F_1=0, F_2=1$. 

%Let $T=T(\delta_1,\delta_2,\delta_3,...)$ be the tree according to \cref{constr:T}.
Let $T_n,n\ge 0$ be the sequence of trees according to \cref{constr:T_n}.
%The sequence~is~BS-convergent to a unimodular tree $\mathcal T$.
%Since $\delta_{n+1} =\delta_n\delta_{n-1}\le \delta_n^2$ for $n\ge 2$, the sequence satisfies \cref{res:conditions_in_d}~\itm1, and 
By \cref{converge} $(T_n,o_n)$ converges in the Benjamini-Schramm sense to a random rooted tree $(\mathcal T,\omega)$ of uni\-form volume growth. 
%
%Since $\delta_n$ is unbounded, the volume of $T$ grows super-polynomially.
%One can~furthermore verify that $|T_n|=(\delta_1+1)\cdots(\delta_n+1)$ approximates $g(2^n)$ where
From $|T_n|=(\delta_1+1)\cdots(\delta_n+1)$ one estimates
$$\tfrac12 D^{r^\alpha}\!\! \le |T_n|\le \tfrac 12r\cdot D^{r^\alpha}\!\!\!,\;\;\text{with $D:=2^{\varphi^2/\sqrt 5}\approx 2.251$ and $\alpha:=\log\varphi\approx 0.6942$}.$$
Here $\varphi\approx 1.618$ denotes the golden ratio.
The growth is therefore \emph{intermediate}.
\end{example}

We claim that $(\mathcal T, \omega)$ is \as\ isomorphic to a particular deterministic tree.
By \cref{structure} it is \as\ an apocentric limit of the $T_n$.
But as we show below (\cref{res:highly_symmetric}), the $T_n$ are highly symmetric in~that any two apocentric copies $\tau,\tau'\prec_{n-1} T_n$ are in fact indistinguishable by symmetry. 
In consequence, there exists (up to symmetry) only one possible inclusion chain leading to an apocentric limit, and $\mathcal T$ is the unique tree obtained in this way.

In the following we establish the symmetry of $T_n$ in an even stronger form:

%As we shall see, this is because the $T_n$ are~highly symmetric, and the BS-limit therefore \enquote{has no choice} but to converge to one single tree.
%We characterize the symmetry of $T_n$:

\begin{lemma}\label{res:highly_symmetric}
%If $\delta_{n+1}=\delta_{n-1}\delta_n$ for $n\ge 1$, then 
With the sequence of \cref{ex:deterministic} the trees $T_n$ are transitive on~apo\-centric copies $\tau\prec_m T_n$ for all $m\le n$. This means, for any two apocentric copies $\tau_1,\tau_2\prec_m T_n$ there is a symmetry $\phi\in\Aut(T_n)$ sending $\phi(\tau_1)=\tau_2$.
\end{lemma}

\begin{proof}
The proof is by induction on $n$. The cases $n\in\{0,1\}$ are trivially verified~and form the induction basis. Suppose that the statement is proven up to some $n\ge 1$; we prove the statement for $T_{n+1}$.

Let $\tau_0\prec_n T_{n+1}$ denote the central copy, and let $\tau_1,...,\tau_{\delta_{n+1}} \prec_n T_{n+1}$ denote the apocentric copies.
From our specific choice of sequence follows that the apocentric copies $\prec_n T_{n+1}$ are in one-to-one relation with the apocentric copies $\prec_{n-2} \tau_0$: the expected number of outwards adjacencies of an apocentric $\hat\tau\prec_{n-2}\tau_0$ is
$$
\frac{\delta_1\cdots\delta_{n-2}}{\delta_1\cdots\delta_n} \cdot \delta_{n+1} 
= \frac{\delta_1\cdots\delta_{n-2}}{\delta_1\cdots\delta_n} \cdot \delta_{n-1}\delta_n = 1,
$$
and by the \enquote{maximally uniform distribution of adjacencies} (\cf\ \cref{rem:max_uniform}, especially~equation \eqref{eq:max_uniform}) the actual number of adjacencies is therefore exactly one.
In other words, there are exactly $\delta_{n+1}$ apocentric copies $\hat\tau_1,...,\hat\tau_{\delta_{n+1}}\prec_{n-2} \tau_0$ and they are in one-to-one relation with the $\tau_i$.
We can assume an enumeration so that $\hat\tau_i$ and $\tau_i$ are connected by an edge, whose end in $\hat\tau_i$ we call $x_i$. %and that $x_i\in\hat\tau_i$ is the \enquote{outwards vertex} adjacent to $\tau_i$.

Fix two apocentric copies $\rho_1,\rho_2\prec_m T_{n+1}$ for some $m\le n+1$.
The goal is to construct a symmetry of $T_{n+1}$ that sends $\rho_1$ onto $\rho_2$. 

Let $i_k\in\NN$ be indices so that $\rho_k\prec_m \tau_{i_k}$.
By induction hypothesis there is a symmetry $\phi_0\in\Aut(\tau_0)$ that maps $\hat\tau_{i_1}$ onto $\hat\tau_{i_2}$. More generally, $\phi_0$ permutes the trees $\hat\tau_1,...,\hat\tau_n$, that is, $\phi_0$ sends $\hat\tau_i$ onto $\hat\tau_{\sigma(i)}$ for some permutation $\sigma\in\Sym(n)$.
It holds $\sigma(i_1)=i_2$.
Note however that $\phi_0$~does not necessarily map $x_i$ onto $x_{\sigma(i)}$. %\enquote{outwards vertices} onto \enquote{outwards vertices}.
We can fix this: invoking the induction hypothesis again, for each $i\in\{1,...,n\}$ exists a symmetry $\phi_1^i\in\Aut(\hat\tau_i)$ that sends $\phi_0(x_i)$ onto $x_{\sigma(i)}$. The map
$$
\phi_2(x):=\begin{cases}
(\phi_1^i\circ\phi_0)(x) & \text{if $x\in \hat\tau_i$} \\
\phi_0(x) & \text{otherwise}
\end{cases}.
$$
is then a symmetry of $\tau_0$ that sends the pair $(\hat\tau_i,x_i)$ onto $(\hat\tau_{\sigma(i)},x_{\sigma(i)})$.
Having this, $\phi_2$ extends to a symmetry $\phi_3\in\Aut(T_{n+1})$ that necessarily sends $\tau_i$ onto $\tau_{\sigma(i)}$.
In particular, it sends $\tau_{i_1}$ (which contains $\rho_1$) onto $\tau_{i_2}$ (which contains $\rho_2$).
As before, $\phi_3$ does not necessarily send $\rho_1$ onto $\rho_2$ right away; but also as before, this can be fixed: by induction hypothesis there is a symmetry $\phi_4\in\Aut(\tau_{i_2})$ that sends $\phi_3(\rho_1)$ onto $\rho_2$. The map
$$
\phi_5(x):=\begin{cases}
(\phi_4\circ\phi_3)(x) & \text{if $x\in \tau_{i_1}$} \\
\phi_3(x) & \text{otherwise}
\end{cases}.
$$
is a symmetry of $T_{n+1}$ that sends $\rho_1$ onto $\rho_2$.
\end{proof}

% \begin{theorem}
%     $\mathcal T$ is deterministic.
% \end{theorem}

% \begin{proof}
%     Since $\delta_n=\Omega(n)$, $\mathcal T$ is a generalized canopy tree, \ie\ $\mathcal T$ is supported exactly on the limits of completely apocentric inclusion chains (\cf\ \cref{rem:general_limit_growth}).
%     Note however that up to symmetry there is only one such chain: since $T_{n+1}$ is transitive on its apocentric copies $\prec_{n} T_{n+1}$, there is (up to symmetry) no choice for the apocentric identification $T_n\subset T_{n+1}$. 
%     Since there is no choice involved, $\mathcal T$ is supported on a single tree.
% \end{proof}

Unimodular random rooted trees that are \as\ isomorphic to a unique tree of even smaller uniform growth can be construct\-ed by setting $\delta_{n+1}:=\delta_n\delta_{n-1}$ for only some $n$,\nlspace and $\delta_{n+1}:=\delta_n$ otherwise.
However, since Benjamini-Schramm limits of mixed sequences have measure zero on every countable set of trees, this approach cannot yield examples with uniform growth below $r^{\log\log r}$.
One might wonder whether this is an artifact of our construction or a general phenomenon (see also \cref{q:deterministic_unimodular}).

%A slight modification of the above example in which we define $\delta_{n+1}:=\delta_n\delta_{n-1}$ for some $n$, and $\delta_{n+1}:=\delta_n$ for others, yields deterministic unimodular random trees of a wide range of uniform volume growth above $r^{\log\log r}$.

\section{Concluding remarks and open questions}
\label{sec:last}

Given a suitable function \gFun, we constructed deterministic trees as well as unimodular random rooted trees of uniform volume growth $g$.
We conclude this article with some open question.

\subsection{Subgraphs of uniform growth}

Our initial approach for constructing trees of uniform intermediate growth (which is not part of this paper) was to start from just any graph of intermediate growth (such as a Cayley graph of the Grigorchuk group \cite{grigorchuck}), and extract a spanning tree that inherits this growth in some way.\nls
Ironically, working out the details of this extraction led to an understanding of~the~desired trees that allowed us constructing them without a need for the ambient graph.
Still, the question remained:

%Having shown that trees of intermediate growth exist and that their growth can be determined at high granularity, the next natural step is to identify such trees in other objects. As discussed in Section 2, the trees of \cref{constr:T} occasionally span graphs which are ``typical'' for their growth. It would be interesting to determine whether this is by necessity:

\begin{question} \label{q:uniform_spanning_tree}
Given a graph $G$ of uniform growth $g$, is there a spanning tree (or just any embedded tree) of the same uniform growth?
%Let $G$ be a graph of uniform growth $g(r)$. Is there a tree embedded in $G$ with the same uniform growth? In particular, is there a spanning tree of $G$ with the same uniform growth? 
\end{question}

%This question can be modified, weakened, and generalized in various forms:

More generally one can ask

% \begin{question}
% Given a graph $G$ of uniform growth $g$, does it contain a planar~subgraph of the same uniform growth?
% \end{question}

% \begin{question}
% Fix $k\in\NN$.
% Given a graph $G$ of uniform growth $g$, does $G$ contain a subgraph $H$ that is \emph{not} $k$-connected and has the same uniform growth $g$?
% %Let $G$ be a graph with uniform growth $g(r)$. Is there some $k\in\mathbb{N}$, perhaps dependent on $g(r)$, such that $G$ has a subgraph which is not $k$-connected and has the same uniform growth?  
% \end{question}

\begin{question}
If $G$ is of uniform growth $g$, and $h\in O(g)$ (perhaps assuming some niceness conditions such as super-additivity and log-concavity), is there a subgraph $H\subset G$ of uniform growth $h$?
\end{question}

%For examples of graphs that do not have subtrees or planar subgraphs with the same macroscopic properties, see \cite{souvlaki}. 

As shown below, an embedded tree $T\subset G$ of the same uniform growth (as we ask for in \cref{q:uniform_spanning_tree}) is not necessarily quasi-isometric to $G$ (in the sense~of~\cite[Chapter 5]{loh2017geometric}).

\begin{example}
Consider a family of cycles chained together to form the graph~$G$~as depicted in \cref{fig:chain_of_cycles}.
%This is more general than requesting a tree quasi-isometric to $G$. 
%Indeed, although quasi-isometry implies agreement of growth (up to a multiplicative constant), one can construct a counterexample to the converse by taking a set of cycles of unbounded length and connecting them serially with edges. 
$G$ is the homeomorphic image of the infinite path with fibers~of size at most two, hence of linear growth.
Removing one edge from each cycle~yields a spanning tree of linear growth.
However, $G$ has infinite perimeter, and hence no quasi-isometric spanning tree. 

\begin{figure}[h!]
    \centering
    \includegraphics[width=0.65\textwidth]{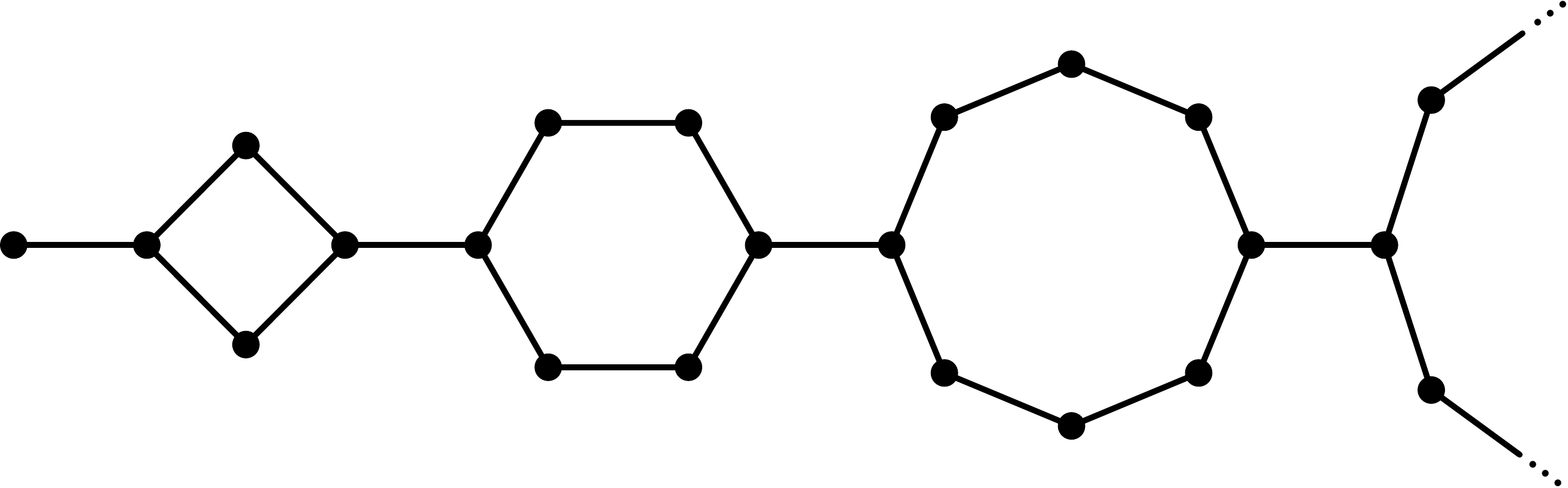}
    \caption{A graph with a spanning tree of the same uniform growth~but with no quasi-isometric spanning tree.}
    \label{fig:chain_of_cycles}
\end{figure}
\end{example}

\subsection{Beyond the construction} 

The unimodular random rooted graphs of uniform volume growth constructed in \cref{sec:unimodular} are Benjamini-Schramm limits of the particular sequence $T_n$. 
We found a change in structure roughly at growth $r^{\log\log r}$ and it is not clear to us whether this is an artifact of our construction or if it points to a more fundamental phase change phenomenon in~unimodular trees of uniform growth.

\begin{question}
To what extent are unimodular trees with growths on either side~of the threshold $r^{\log\log r}$ structurally different?
\end{question}

In light of \cref{ex:deterministic} and the discussion below it we also wonder the following:

\begin{question}\label{q:deterministic_unimodular}
Is there a random rooted tree $(\mathcal T,\omega)$ that satisfies simultaneously~all of the following three properties?:
\begin{myenumerate}
    \item $(\mathcal T,\omega)$ is unimodular.
    \item $(\mathcal T,\omega)$ is \as\ isomorphic to a particular deterministic tree.
    \item $(\mathcal T,\omega)$ is of super-linear uniform volume growth $g\in O(r^{\log\log r})$.
\end{myenumerate}
\end{question}

%\todo{What about the limit of a sequence that is eventually 1 but not initially 1?}

Questions about threshold phenomena and essentially deterministic unimodular graphs are of course equally meaningful for graph classes other than trees.

\bibliographystyle{abbrv}
\bibliography{literature}

\begin{thebibliography}{10}

\bibitem{abert2022co}
M.~Abert, M.~Fraczyk, and B.~Hayes.
\newblock Co-spectral radius, equivalence relations and the growth of
  unimodular random rooted trees.
\newblock {\em arXiv preprint arXiv:2205.06692}, 2022.

\bibitem{ambj1997quantum}
J.~Ambj, J.~Ambj{\o}rn, B.~Durhuus, T.~Jonsson, O.~Jonsson, et~al.
\newblock {\em Quantum geometry: a statistical field theory approach}.
\newblock Cambridge University Press, 1997.

\bibitem{amir2022branching}
G.~Amir and S.~Yang.
\newblock The branching number of intermediate growth trees.
\newblock {\em arXiv preprint arXiv:2205.14238}, 2022.

\bibitem{angel2003growth}
O.~Angel.
\newblock Growth and percolation on the uniform infinite planar triangulation.
\newblock {\em Geometric And Functional Analysis}, 13(5):935--974, 2003.

\bibitem{babaigrowth}
L.~Babai.
\newblock The growth rate of vertex-transitive planar graphs.
\newblock {\em SODA '97: Proceedings of the eighth annual ACM-SIAM symposium on
  Discrete algorithms}, pages 564--573, 1997.

\bibitem{benjamini2021triangulations}
I.~Benjamini and A.~Georgakopoulos.
\newblock Triangulations of uniform subquadratic growth are quasi-trees.
\newblock {\em Annales Henri Lebesgue}, 5:905--919, 2022.

\bibitem{benjamini2011recurrence}
I.~Benjamini and O.~Schramm.
\newblock Recurrence of distributional limits of finite planar graphs.
\newblock In {\em Selected Works of Oded Schramm}, pages 533--545. Springer,
  2011.

\bibitem{curien}
N.~Curien.
\newblock {\em Random graphs}.
\newblock 2017.

\bibitem{grigorchuck}
R.~I. Grigorchuck.
\newblock Degrees of growth of finitely generated groups and the theory of
  invariant means.
\newblock {\em Izvestiya Akademii Nauk SSSR. Seriya Matematicheskaya},
  48(5):939--985, 1984.

\bibitem{loh2017geometric}
C.~L{\"o}h.
\newblock {\em Geometric group theory}.
\newblock Springer, 2017.

\bibitem{aldouslyons}
R.~Lyons and D.~Aldous.
\newblock Processes on unimodular random networks.
\newblock {\em Electronic Journal of Probability}, 12(54):1454--1508, 2007.

\bibitem{timar2014stationary}
{\'A}.~Tim{\'a}r.
\newblock A stationary random graph of no growth rate.
\newblock In {\em Annales de l'IHP Probabilit{\'e}s et statistiques},
  volume~50, pages 1161--1164, 2014.

\bibitem{trofimov1985graphs}
V.~I. Trofimov.
\newblock Graphs with polynomial growth.
\newblock {\em Mathematics of the USSR-Sbornik}, 51(2):405, 1985.

\end{thebibliography}

\newpage
\appendix

\section{Existence and properties of $\gamma$}
\label{sec:gamma_properties}

Recall the function $\gamma\:\NN\to\RR$ introduced in the proof of \cref{res:conditions_in_g}, defined by $\gamma(1):=25$ and
\begin{equation}\label{eq:gamma}
 \gamma(\delta):=\prod_{k=0}^\infty \Big(1+\frac2{\delta^{2^k}}\Big)^{\!2},\quad\text{whenever $\delta\ge 2$}.   
\end{equation}
It is evident from the definition that $\gamma(\delta)> 1$. % and $\gamma$ is decreasing in its argument (except $\gamma(1)>\gamma(2)$, which is shown in \cref{res:gamma_bounds})
We show that it is well-defined:

\begin{lemma}\label{res:estimate_gamma}
$\gamma(\delta)$ is well-define, \ie\ the infinite product \eqref{eq:gamma} converges. % for $\delta\ge 2$.
%For $\delta\geq 2$, $\gamma(\delta):=\!\prod_{k\ge 0} \!\big(1+2 / \delta^{\kern0.1ex2^k} \big)^2\!$ converges.
%In particular, we have the following upper bounds: $\gamma(2)<12.33$, $\gamma(3)<5$, $\gamma(5)<4$, $\gamma(7)<3$, and $\gamma(9)<2$. %\nlspace
%
%$$
    %\gamma(2) \le 11.33 
    %\quad
    %\text{ and } 
    %\quad
    %\gamma(3) \le 5.
%$$
\end{lemma}

\begin{proof}
$\gamma(\delta)$ is the square of
 \[
 \prod_{k=0}^\infty\Big(1+\frac{2}{\delta^{2^k}}\Big)
 =\sum_{i=0}^\infty\frac{2^{b(i)}}{\delta^i},
 \] 
 where $b(i)$ denotes the number of 1's in the binary representation of $i$. Since $b(i)\leq \lfloor\log_2(i)\rfloor+1\leq \frac{i}{2}$ for $i\geq 10$, a geometric series estimation yields 
 \begin{equation}\label{eq:gamma_upper_bound}
 \prod_{k=0}^\infty\Big(1+\frac{2}{\delta^{2^k}}\Big)
 \,\leq\, \sum_{i=0}^9\frac{2^{b(i)}}{\delta^i}
    +\sum_{\mathclap{i=10}}^\infty\frac{\sqrt 2^{\kern0.1ex i}}{\delta^i}
 \,\leq\, \sum_{i=0}^9\frac{2^{b(i)}}{\delta^i}
    +
    %\underbrace{
    %\frac{1}{32}\frac{\sqrt{2}}{\sqrt{2}-1}
    \frac{2+\sqrt 2}{32}
    %}_{\approx\, 0.106694...}
    ,
 \end{equation}
which is finite. 
%
%The given upper bounds can be easily verified.
\end{proof}

% \begin{align*}
% \begin{array}{rcl}
%     11 <& \!\! \gamma(2) \!\! &< 12 \\
%     4 <& \!\! \gamma(3) \!\! &< 5 \\
%     2 <& \!\! \gamma(4) \!\! &< 3
% \end{array}
% \end{align*}

\begin{corollary}\label{res:gamma_bounds}
The following bounds apply:
\begin{align*}
\arraycolsep=0.5ex
\begin{array}{rcccl}
9&<&\gamma(2)&<&25/2 \\
25/9&<&\gamma(3)&<&5.
\end{array}
\end{align*}
\begin{proof}
The upper bounds can be verified from the proof of \eqref{eq:gamma_upper_bound}.
The lower bounds follow via
\begin{align*}
\gamma(2) &> \Big(1+\frac2{2}\Big)^2\Big(1+\frac24\Big)^2 = 9, \\
\gamma(3) &> \Big(1+\frac23\Big)^2 = \frac{25}9.
\end{align*}
\end{proof}
\end{corollary}

\begin{corollary}\label{res:gamma_decreasing}
$\gamma(\delta)$ is strictly decreasing in $\delta$.
\begin{proof}
For $2\le \delta_1<\delta_2$ follows $\gamma(\delta_1)<\gamma(\delta_2)$ straight from the definition.
For~$\delta_1=1$ we use the bounds in \cref{res:gamma_bounds} and find
$\gamma(1) = 25 > 25/2 > \gamma(2).$
\end{proof}
\end{corollary}

\begin{proposition}\label{res:d_gamma_monotone}
$\delta^2\gamma(\delta)$ is strictly increasing for $\delta\ge 3$.
\begin{proof}
The function in question is the square of
$$f(\delta):=\delta \prod_{k=0}^\infty \Big(1+\frac2{\delta^{2^k}}\Big).$$
If we consider $\delta$ as a continuous parameter, we can compute the derivative of $\log f(\delta)$ \wrt\ $\delta$, which yields
$$(\log f(\delta))'=\frac1\delta - \frac2\delta\sum_{k=0}^\infty \frac{2^k}{\delta^{2^k}+2}.$$
For $\delta^2 \gamma(\delta)$ to be increasing, we require this expression to be positive, which, after some rearranging, is equivalent to
$$\sum_{k=0}^\infty\frac{2^k}{\delta^{2^k}+2}<\frac12.$$
The sum is clearly decreasing in $\delta$ and so it suffices to verify the case $\delta=3$:
\begin{align*}
\sum_{k=0}^\infty \frac{2^k}{3^{2^k}+2} 
&=\frac15+\frac29 + \sum_{k=2}^\infty \frac{2^k}{3^{2^k}+2}
\\&<\frac{21}{55} + \sum_{k=2}^\infty \frac{2^k}{3^{2^k}}
<\frac{21}{55} + \sum_{\ell=4}^\infty \frac{\ell}{3^{\ell}} = \frac{21}{55}+\frac1{12}<\frac12.
\end{align*}
\end{proof}
\end{proposition}

% \begin{lemma}\label{res:gamma_property_d_d2}
% For $\delta_1\le\delta_2$ holds $\delta_1\gamma(\delta_1)\le \delta_2^2\gamma(\delta_2)$.
% %
% \begin{proof}
% The statement is clear for $\delta_1=\delta_2$, so we assume $\delta_1<\delta_2$.
% We proceed by case analysis.
% We repeatedly invoke the bounds from \cref{res:gamma_bounds} as well as the monotonicity result \cref{res:d_gamma_monotone}, the latter we indicate by $(*)$.
% %
% \begin{enumerate}[label=(\arabic*),topsep=0.7em]
%     \item if $\delta_1\ge 3$, then
%     %
%     $$\delta_1\gamma(\delta_1)<\delta_1^2\gamma(\delta_1) \overset{\smash{(*)}}\le \delta_2^2\gamma(\delta_2).$$
    
%     \item if $\delta_1 = 2$, then $\delta_2\ge 3$ and 
%     %
%     $$\delta_1\gamma(\delta_1)=2\gamma(2) < 25 < 3^2\gamma(3) \overset{\smash{(*)}}\le \delta_2^2\gamma(\delta_2).$$
    
%     \item if $\delta_1 = 1$ and $\delta_2\ge 3$ then 
%     %
%     $$\delta_1\gamma(\delta_1)=1\gamma(1)= 25 < 3^2\gamma(3) \overset{\smash{(*)}}\le \delta_2^2\gamma(\delta_2).$$
    
%     \item if $\delta_1 = 1$ and $\delta_2 = 2$ then 
%     %
%     $$\delta_1\gamma(\delta_1)=1\gamma(1)= 25 < 36 < 2^2\gamma(2).$$
% \end{enumerate}
% \end{proof}
% \end{lemma}

The following property of $\gamma$ is used in the proof of \cref{res:conditions_in_g}.

\begin{lemma}\label{res:gamma_property_d_d2}
If $\delta_1\le(\delta_2+2)^2$ then $\delta_1\gamma(\delta_1)\le \delta_2^2\gamma(\delta_2)$.
\begin{proof}
We proceed by case analysis. A use of \cref{res:gamma_decreasing} (\ie\ $\gamma$ is decreasing) is indi\-cated by $(*)$, a use of \cref{res:d_gamma_monotone} (\ie\ $\delta^2\gamma(\delta)$ is increasing) by $(**)$. We also use the bounds proven in \cref{res:gamma_bounds}. 
\begin{enumerate}[topsep=0.7em,itemsep=0.7em]
\item
If $\delta_1\ge \delta_2^2$ then
\begin{align*}
    \delta_1\gamma(\delta_1) 
      &\overset{\mathclap{\smash{(*)}}}\le (\delta_2+2)^2 \gamma(\delta_2^2)
       = \delta_2^2 \Big(1+\frac2{\delta_2}\Big)^2 \prod_{k=0}^\infty \Big(1+\frac2{(\delta_2^2)^{2^k}}\Big)^2 
    \\[-1.5ex]&\!= \delta_2^2 \prod_{k=0}^\infty \Big(1+\frac2{\delta_2^{2^k}}\Big)^2
       \!= \delta_2^2 \gamma(\delta_2).
\end{align*}

\item
If $\delta_2 \le \delta_1< \delta_2^2$ then $\delta_1\gamma(\delta_1) \overset{\mathclap{\smash{(*)}}}< \delta_2^2\gamma(\delta_2)$ follows by factor-wise comparison

\item If $3\le\delta_1\le\delta_2$ then
$$\delta_1\gamma(\delta_1)<\delta_1^2\gamma(\delta_1) \overset{\smash{(**)}}\le \delta_2^2\gamma(\delta_2).$$

\item If $\delta_1 = 2$ and $\delta_2\ge 3$ then
$$\delta_1\gamma(\delta_1)=2\gamma(2) < 25 < 3^2\gamma(3) \overset{\smash{(**)}}\le \delta_2^2\gamma(\delta_2).$$

\item If $\delta_1 = 1$ and $\delta_2\ge 3$ then
$$\delta_1\gamma(\delta_1)=1\gamma(1)= 25 < 3^2\gamma(3) \overset{\smash{(**)}}\le \delta_2^2\gamma(\delta_2).$$

\item if $\delta_1 = 1$ and $\delta_2=2$ then
$$\delta_1\gamma(\delta_1)=1\gamma(1)= 25 < 36 < 2^2\gamma(2).$$
\end{enumerate}
\end{proof}
\end{lemma}

\iffalse % aux lemma for Theorem 3

\newpage

\begin{lemma}
    For $r, n$ large enough and $\delta_n$ slow-growing, only $k$-apocentric vertices are the centers of $k$-apocentric $r$-balls in $T_n$.
\end{lemma}

\begin{proof}
    The centers of $k$-apocentric $r$-balls have vertices of degree $\delta_1+1$ at all odd distances. The vertices that are central in their copies of $T_1$ have vertices of degree $\delta_1+1$ at all even distances, so they cannot be centers of $k$-apocentric balls, since that would yield a fast-growing sequence. In particular, the vertices that are central in their copies of $T_1$, and therefore said copies, can be identified by local examination.

    Inductively, suppose that the copies of $T_m$, $m<k$ have been identified, and consider the corresponding auxiliary tree. The copies of $T_m$ containing centers of $k$-apocentric $r$-balls have copies of degree $\delta_{m+1}+1$ at all odd distances. The copies of $T_m$ that are central in their copy of $T_{m+1}$ have copies of degree $\delta_{m+1}+1$ at all even distances in the auxilliary graph, so they cannot contain the centers of $k$-apocentric balls, since that would yield a fast-growing sequence. This also implies that the copies of $T_m$ that are central in their copy of $T_{m+1}$, and therefore the copies of $T_{m+1}$, can be identified by local examination. This concludes the inductive step, and the proof.
\end{proof}

\fi

\end{document}